\documentclass[10pt]{amsart}

\usepackage[english,francais]{babel}

\def\Cset{\mathbf{C}}
 \def\Hset{\mathbf{H}}
 
 \def\Qset{\mathbf{Q}}
 \def\Rset{\mathbf{R}}
 \def\Zset{\mathbf{Z}}
 \def\Tset{\mathbb{T}}
 
 \def\Fset{\mathbb{F}}

\def\Kset{\mathbf{K}}

\def\wt{\widetilde}

\def\leq{\leqslant }
\def\geq{\geqslant}

\usepackage{times,amsfonts,amsmath,amstext,amsbsy,amssymb,
  amsopn,amsthm,upref,eucal, amscd}
\usepackage[T1]{fontenc}

\usepackage[pdftex]{graphicx}

\newtheorem{theorem}{Theorem}[section]
\newtheorem{lemma}[theorem]{Lemma}
\newtheorem{corollary}[theorem]{Corollary}
\newtheorem{proposition}[theorem]{Proposition}

\numberwithin{equation}{section}

\theoremstyle{definition}
\newtheorem{definition}[theorem]{Definition}
\newtheorem{example}[theorem]{Example}

\newtheorem{remark}[theorem]{Remark}

\begin{document}

\title[Homology of origamis with symmetries]{Homology of origamis with symmetries \\ (Homologie des origamis avec sym\'etries)}

\author[C. Matheus]{Carlos Matheus}
\address{ Universit\'e Paris 13, Sorbonne Paris Cit\'e, LAGA, CNRS (UMR 7539), F-93430, Villetaneuse, France.}
\email{ matheus.cmss(at)gmail.com}
\author[J.-C. Yoccoz]{Jean-Christophe Yoccoz}
\address{Coll\`ege de France (PSL), 3, Rue d'Ulm, 75005
Paris, France}
\email{ jean-c.yoccoz(at)college-de-france.fr }
\author[D. Zmiaikou]{David Zmiaikou}
\address{D\'epartement de Math\'ematiques, Universit\'e Paris-Sud 11, 91405 Orsay Cedex, France}
\email{ david.zmiaikou(at)gmail.com}

\selectlanguage{english}

\keywords{Origamis, square-tiled surfaces, automorphisms group, affine group, representations of finite groups, regular and quasi-regular origamis, Kontsevich-Zorich cocycle, Lyapunov exponents. \\ .\hspace{0.2cm} \emph{Mots-cl\'es.} Origamis, surfaces \`a petits carreaux, groupes d'automorphismes, groupes affines, repr\'esentations des groupes finis, origamis r\'eguliers et quasi-r\'eguliers, cocycle de Kontsevich-Zorich, exposants de Lyapunov.}


\date{\today}

\maketitle


\begin{abstract}
Given an origami (square-tiled surface) $M$ with automorphism group $\Gamma$, we compute the decomposition of the first homology group of $M$ into isotypic
$\Gamma$-submodules. Through the action of the affine group of $M$ on the homology group, we deduce some consequences for the multiplicities of the
Lyapunov exponents of the Kontsevich-Zorich cocycle. We also construct and study several families of interesting origamis illustrating our results.
\end{abstract}

\selectlanguage{francais}
\begin{abstract} \'Etant donn\'e un origami (surface \`a petits carreaux) $M$ avec un groupe d'automorphismes $\Gamma$, nous d\'eterminons la d\'ecomposition du premier groupe d'homologie de $M$ en $\Gamma$-submodules isotypiques. Parmi l'action du groupe affine de $M$ sur le groupe d'homologie, nous d\'eduisons quelques cons\'equences pour les multiplicit\'es des exposants de Lyapunov du cocycle de Kontsevich-Zorich. De plus, nous construisons et \'etudions plusieurs familles d'origamis int\'eressants pour illustrer nos r\'esultats. 
\end{abstract}

\selectlanguage{english}

\tableofcontents

\section{Introduction}\label{intro}

Consider a (compact, connected) Riemann surface $M$ equipped with a non-zero holomorphic $1$-form $\omega$. The local primitives of  $\omega$ form an atlas on $M - \{\textrm{zeros of } \omega\}$
for which the coordinate changes are given locally by translations, hence the name {\it translation surface} for the data $(M,\omega)$. These geometric structures and their moduli spaces have attracted a lot of attention since
 the pioneering work of Masur (\cite {Masur1}) and Veech (\cite {Veech1}) in the early 1980's. Translation surfaces occur naturally in the study of billiards in rational polygons, and
as suspension of interval exchange transformations.

\medskip

Given a translation surface $(M,\omega)$ and an element $g \in SL(2,\Rset)$, one obtains an atlas which defines another structure of translation surface on $M$ by postcomposing the preferred charts by $g$.
One defines in this way an action of $SL(2,\Rset)$ on the moduli space of translation surfaces. The restriction of this action to the $1$-parameter diagonal subgroup $\textrm{diag}(e^t,e^{-t})$
is the  {\it Teichm\"uller flow} on the moduli space. It can also be seen as a suspension of the Rauzy-Veech renormalization algorithm for interval exchange transformations (\cite {Veech}, \cite {Y-Pisa}, \cite {Zorich6}) .

\medskip

The {\it Veech group}  $SL(M,\omega)$ of a translation surface $(M,\omega)$ is the stabilizer for this action of the point represented by $(M,\omega)$ in moduli space.
 The {\it affine group} $\textrm{Aff}_+ (M,\omega)$ is the group of orientation-preserving homeomorphisms of $M$ which preserve the set of zeros of $\omega$ and can be read
  locally in the preferred charts as  affine self-maps of the plane. Similarly, the automorphism group $\textrm{Aut} (M,\omega)$ is the group of
   orientation-preserving homeomorphisms of $M$ which preserve the set of zeros
 of $\omega$ and can be read locally in the preferred charts as translations of the plane. Given an affine homeomorphism of $(M,\omega)$, its derivative in the preferred charts
 is constant
 and this constant value belongs to the Veech group. We actually have an exact sequence
  $$1\to\textrm{Aut}(M)\to \textrm{Aff}_+(M)\to SL(M)\to 1,$$
 where, as in the rest of the paper, we have written $M$ instead of  $(M,\omega)$ when the translation structure is clear from the context.

 \medskip

A {\it Veech surface} is a translation surface whose Veech group is a lattice in $SL(2,\Rset)$. In this paper, we will be interested in the simplest class of Veech surfaces: they go
in the litterature by various names,  {\it arithmetic Veech surfaces, square-tiled surfaces or origamis} , and can be characterized in several equivalent ways. For instance,
they are those translation surfaces whose Veech group is commensurate to $SL(2,\Zset)$; they are also (from a theorem of Gutkin-Judge \cite {GJ}) those translation
surfaces $(M,\omega)$ such that there exists a
covering map $\pi: M \to \Tset^2 = \Cset/ \Zset^2$ unramified off $0$ with $\pi^*(dz) =\omega$. Another more algebraic viewpoint has been developed  by
Herrlich, Schmith\"usen and Zmiaikou (\cite {He} \cite {HS} \cite {Zm}):  an origami is a finite set of copies of the standard square with the left (resp. top) side of every square
identified with the right (resp. bottom) side of some another (possibly the same) square.

 \medskip

 The Kontsevich-Zorich cocycle (KZ cocycle) is a cocycle over the Teichm\"uller flow first considered in
   \cite{K}, \cite{Zorich1}, \cite{Zorich2}, \cite{Zorich3}, \cite{Zorich4}, \cite{Zorich5}
 which has been the subject of several important works in the last few years. Roughly speaking,
 the KZ cocycle is the Gauss-Manin connection acting on the first homology (or cohomology) group. Above the $SL(2,\Rset)$-orbit of a surface with a non trivial group
 of automorphisms, the definition requires some extra care and is given in Section $4$. Zorich discovered that the Lyapunov exponents of this symplectic cocycle with respect to
 ergodic $SL(2,\Rset)$-invariant probability measures on moduli space govern  the \emph{deviations of ergodic averages} of  typical
 interval exchange transformations, translation flows and billiards. The knowledge of some of these exponents  was recently exploited by  Delecroix, Hubert and  Leli\`evre \cite{DHL}
 to confirm the prediction of the physicists J. Hardy and J. Weber of  unexpected speed of diffusion of trajectories in ``typical realizations'' of the periodic Ehrenfest's wind-tree model
 for Lorenz gases.

 \medskip

 Very different methods have been used to understand the Lyapunov exponents of the KZ-cocycle.  Kontsevich and Zorich had conjectured that the Lyapunov exponents
 with respect to the Masur-Veech measures on connected component of strata of
   moduli space are all simple. This has been   fully proved by A. Avila and M. Viana \cite{AV} by using \emph{combinatorial} properties of the so-called Rauzy-Veech induction,
   after an important partial answer by G. Forni \cite{Fo} based on \emph{complex analytical} and \emph{potential theoretical} methods.
Several authors have used \emph{algebro-geometric} methods to compute individual values and/or sums of these exponents with respect to  ergodic
$SL(2,\Rset)$-invariant probability measures:
 Bainbridge \cite{Ba},  Bouw and M\"oller \cite{BM}, Chen and M\"oller \cite{CM}, Yu and Zuo \cite{YZ}, and most importantly Eskin,  Kontsevich and Zorich \cite{EKZ1}, \cite{EKZ2}.

\medskip

The purpose of this paper is to study through elementary representation theory the relation between the automorphism group $\Gamma$ of an origami $M$ and the
Lyapunov exponents of the KZ-cocycle over the
$SL(2,\Rset)$-orbit of $M$ in moduli space (equipped with the probability measure derived from Haar measure on  $SL(2,\Rset)$). The relation is given by the action of the
affine group on homology. On one hand, the Lyapunov exponents can be viewed as limits
\begin{equation*}
 \theta = \lim_{n \rightarrow + \infty} \frac {\log ||A_n(v)||}{\log ||A_n||}
 \end{equation*}
for some appropriate sequence $(A_n)$ of affine homeomorphisms of $M$. On the other hand,  the homology group has the structure of a $\Gamma$-module, and this
 structure is preserved by a subgroup of finite index of the affine group. It follows that the decomposition of the homology group into isotypic $\Gamma$-submodules is preserved
 by the KZ-cocycle. Moreover, the symplectic structure on the homology group given by the intersection form imposes further restrictions, especially when the isotypic components
 are of complex or quaternionic type.

 \medskip

The paper is organized as follows. Section $2$ introduces our preferred setting for origamis: given a finite group $G$ generated by two elements $g_r, g_u$ and a
subgroup $H$ of $G$ such that the intersection of the conjugates of $H$ in $G$ is trivial, we associate an origami $M$ by taking for the squares of $M$ the right classes $Hg$,
with $Hg g_r$ to the right of $Hg$ and $Hg g_u$ to the top of it. The automorphism group $\Gamma$ is then naturally isomorphic to the quotient by $H$ of the normalizer
$N$ of $H$ in $G$.   In this representation, the fiber $\Sigma^*$ of $0$ for the canonical map $\pi: M \to \Tset^2$ is given by the action of the commutator $c$ of $g_r, g_u$
on the set of squares. In the end of the section, it is explained how the structure of the homology group as a $\Gamma$-module  is related to the action of $\Gamma$
on $\Sigma^*$.

\medskip

In Section $3$, we obtain the decomposition of the homology with complex coefficients into isotypic $\Gamma$-modules. First, the $\Gamma$-module  $ H_1 (M, \Cset) $ splits into
a trivial $\Gamma$-module of complex dimension $2$ denoted $H_1^{st} (M, \Cset) $ and a complementary module denoted $H_1^{(0)} (M,\Cset)$. For any irreducible
representation $\rho_{\alpha}: \Gamma \to \textrm{Aut}(V_{\alpha})$, a formula for the multiplicity $\ell_{\alpha}$ of $\rho_{\alpha}$ in  $H_1^{(0)} (M,\Cset)$ is obtained
in terms of the action of the commutator $c=[g_r,g_u]=g_r g_u g_r^{-1}g_u^{-1}$ in the  representation induced by $\rho_{\alpha}$ to $G$. It shows that this multiplicity is greater or equal than the multiplicity $\ell_0$
 of the trivial representation. In the simplest case of a {\it regular} origami, corresponding to $H=\{1\}$, we have
\begin{theorem}
For the regular origami associated to the generators $g_r, g_u$ of $G$, the multiplicity of the representation  $\rho_{\alpha}$ in $H_1^{(0)} (M,\Cset)$ is equal to the
codimension of the fixed space of $\rho_{\alpha}([g_r,g_u])$.
\end{theorem}
By considering several distinct cases, we prove from the general formula the following interesting fact
\begin{theorem}
For any origami $M$, the multiplicity  $\ell_{\alpha}$ of any irreducible representation $\rho_{\alpha}$ of $\textrm{Aut}(M)$ in  $H_1^{(0)} (M,\Cset)$ is never equal to $1$.
\end{theorem}
The formula for the multiplicity suggests to consider a class of origamis slightly larger than the class of regular origamis, namely those for which the trivial representation is
not a sub-representation of  $H_1^{(0)} (M,\Cset)$. We obtain a structure theorem for such {\it quasi regular} origamis.
\begin{theorem}
The origami associated to the data $(G,g_r, g_u,H)$ is quasi regular iff $N$ is normal in $G$ and $G/N$ is abelian.
\end{theorem}

\medskip

 The end of Section $3$ is a preparation to the study of the affine group. We consider the homology group  $H_1^{(0)} (M,\Rset)$ with {\bf real} coefficients and
 its decomposition into isotypic components. Each such component $W_a$  is of {\it  real, complex} or {\it quaternionic} type. We discuss in each case the nature of the group
 $Sp(W_a)$ consisting of the  $\Gamma$-automorphisms of $W_a$ which preserve the restriction of the symplectic form. The group $Sp(W_a)$ is isomorphic to a symplectic
 group in the real case, a complex unitary group $U_{\Cset}(p,q)$ in the complex case, a quaternionic unitary group $U_{\Hset}(p,q)$ in the quaternionic case. The material is fairly
 classical and is recalled for the sake of the reader not familiar with classical representation theory.

 \medskip

 In Section $4$, we deduce consequences for the Lyapunov exponents of KZ cocycle over the $SL(2,\Rset)$-orbit of  an origami, equipped with the probability measure derived from
  Haar measure on  $SL(2,\Rset)$. Indeed the decomposition of $H_1$ into $H_1^{st}$ and $H_1^{(0)}$ is preserved by the affine group, and the decomposition of $H_1^{(0)}$
  into isotypic components is preserved by a subgroup of finite index of the affine group. Moreover, this subgroup acts on each component $W_a$ by elements of $Sp(W_a)$.
 This allows to split the nontrivial Lyapunov exponents into subfamilies associated to the $W_a$. Regarding the multiplicity of the Lyapunov exponents in each subfamily, we obtain
 \begin{theorem}
 Let $\rho_a: \Gamma \to V_a$ be an $\Rset$-irreducible representation of $\Gamma$. Let $W_a$ be the associated isotypic  component in  $H_1^{(0)}(M,\Rset)$.
 \begin{enumerate}
 \item Each factor in the Oseledets decomposition of $W_a$ is a $\Gamma$-submodule. Therefore the multiplicity in $W_a$ of each Lyapunov exponent is a multiple
 of the dimension (over $\Rset$) of $V_a$.
 \item When $V_a$ is complex or quaternionic, and $Sp(W_a)$ is isomorphic to an unitary group $U(p,q)$, the multiplicity in $W_a$ of the exponent $0$ is at least
 $|q-p| \dim_{\Rset} V_a$.
 \end{enumerate}
 \end{theorem}

\medskip

In the last section, we construct and study several (families of) examples. We first consider a family of quasi regular origamis, indexed by an integer $n \ge 2$, such that the
automorphism group $\Gamma$ is isomorphic to the symmetric group $S_n$. One interesting feature of these origamis, compared to the regular origamis
based on the symmetric groups,
 is that the commutator $c$ of the two generators of $G$ correspond to a transposition which is not a commutator in $\Gamma$. In the next subsection, we construct a quasi regular
 origami such that the quotient $G/N$ is not cyclic. Then, we consider a family of regular origamis associated to the simple groups of Lie type $SL(2,\Fset_p)$, with a rather natural
 pair of unipotent generators. We compute the multiplicities of the irreducible representations in  $H_1^{(0)}$ from the general formula of Section $3$, but stop short of determining
  the signatures $(p,q)$ in the complex or quaternionic cases.

  \smallskip

  We don't know of any regular origami where the affine group permutes in a nontrivial way the isotypic components of  $H_1^{(0)}(M,\Rset)$. According to Section $4$, this would give even higher multiplicity to the Lyapunov exponents of the KZ-cocycle. We construct in Subsection 5.4 a poor substitute for that example.

\section*{Acknowledgements} The authors are thankful to the Coll\`ege de France (Paris), the Laboratoire Analyse, G\'eom\'etrie et Applications (Universit\'e Paris 13),  the Department of Mathematics of Universit\'e Paris-Sud and  the Instituto de Matem\'atica Pura e Aplicada  (IMPA, Rio de Janeiro)
 for their hospitality and support during the preparation of this article. The  authors were also supported by the French ANR grant ``GeoDyM'' (ANR-11-BS01-0004) and by the Balzan Research Project of J. Palis.

\section{Preliminaries}

\subsection{Origamis}
\begin{definition}
A {\it square-tiled surface} (or {\it origami}) is a ramified covering $\pi: M \rightarrow \Tset^2:=\Rset^2 / \Zset^2$ from a {\bf connected}
 surface to the standard torus which is unramified off $0 \in \Rset^2 / \Zset^2$.
We will denote by $\Sigma^*$ the fiber $p^{-1}(0)$. It contains the set $\Sigma$ of ramification points
of $\pi$.

The {\it squares} are the connected components of the inverse image $\pi^{-1}((0,1)^2)$.
\end{definition}

The set ${\rm Sq}(M)$ of squares is finite and equipped with two one-to-one self maps
$r$ (for right) and $u$ (for up) which associate
to a square the square to the right of it and above it (resp.). The connectedness of the surface means that the group of permutations of
${\rm Sq}(M)$ generated by $r$ and $u$ acts transitively on ${\rm Sq}(M)$. Conversely, a finite set $\mathcal O$, equipped
with two one-to-one maps
$r$  and $u$ such that the group of permutations
 generated by $r$ and $u$ acts transitively on $\mathcal O$, defines a square-tiled surface.

\begin{definition}
An {\it automorphism} of the origami  $\pi: M \rightarrow \Tset^2$ is a homeomorphism $f$ of $M$ which satisfies $\pi \circ f = \pi$. It induces a one-to-one map $f^*$ on ${\rm Sq}(M)$ commuting with $r$ and $u$. Conversely, any such map determines a unique automorphism of $M$. We denote the group of automorphisms of $M$ by ${\rm Aut}(M)$.
\end{definition}

 \smallskip

 The following definitions were introduced in \cite{Zm}, where the corresponding point of view was extensively developed.

 \begin{definition}

 The subgroup of the permutation group $\mathcal S ( {\rm Sq}(M))$ generated  by $r$ and $u$ is called the monodromy group of $M$ and denoted by $\mathcal Mon(M)$.
\end{definition}
Observe that the stabilizers of the points of ${\rm Sq}(M)$ form a conjugacy class of subgroups of $\mathcal Mon(M)$ whose intersection is reduced to the identity.

Conversely, let $G$ be a finite group and $g_r, g_u$ be two elements of $G$ which span $G$. Let $H$ be a subgroup of $G$ which does not contain a normal subgroup of $G$ distinct from the identity. We associate to these data an origami by taking for squares the  right classes mod $H$, and defining
$$ r(Hg) = Hgg_r, \quad u(Hg) =Hgg_u.$$
The monodromy group of this origami is then canonically isomorphic to $G$ and the stabilizer of the class of the identity element is equal to $H$.

Observe that two sets of data $(G,g_r, g_u,H)$, $(G',g'_r, g'_u,H')$ determine isomorphic origamis iff there exists an isomorphism $\Phi$ from $G$ onto $G'$ such that $\Phi(g_r) = g'_r, \Phi(g_u) = g'_u$ and $\Phi(H)$ is conjugate to $H'$.
In particular, the group of automorphisms of the origami determined by $(G,g_r, g_u,H)$ is canonically isomorphic to the quotient $N(H)/H$ of the normalizer $N(H)$ of $H$ by $H$, acting on the left on $\mathcal O$ by
$$n.Hg = Hng.$$

\begin{definition}

An origami $\pi: M \rightarrow \Tset^2$ is \emph{regular} if the automorphism group ${\rm Aut}(M)$  acts transitively on ${\rm Sq}(M)$. This is equivalent to ask that the stabilizer in the monodromy group of some (any) element of ${\rm Sq}(M)$
is reduced to the identity, i.e., we have $H=\{1\}$ in the construction above.

\end{definition}

See \cite{Zm} for some examples of regular origamis.

\medskip

Until the end of Section 4,  we fix an origami $\pi: M \rightarrow \Tset^2$, presented as above from a group $G$, a pair of generators $g_r,g_u$, and a subgroup $H$ which does not contain a normal subgroup distinct from the identity. We thus identify ${\rm Sq}(M)$ with $H \backslash G$.
We denote by $N$ the normalizer of $H$ in $G$ and by $c$ the commutator $c=g_r g_u g_r^{-1}g_u^{-1}$ of the preferred generators of $G$.  The subset $\Sigma^* = p^{-1}(0)$ is identified with the set of orbits for the action of the subgroup $<c>$ generated by $c$ on $H \backslash G$. This action also determines the stratum of the moduli space of translation surfaces containing $M$.

\subsection{Exact sequences for homology groups}

Let $\Kset$ be a subfield of $\Cset$. Consider the $\Kset$-vector space $\Kset(M):= \Kset^{H \backslash G}$ having a canonical basis $(e_s)$ indexed by the squares of $M$. Consider also  the sum $\Kset(M) \oplus \Kset(M)$ of two copies of $\Kset(M)$. We write $(\sigma_s)_{s\in H \backslash G}$ for the canonical basis of the first factor and
$(\zeta_s)_{s \in H \backslash G}$ for the canonical basis of the second factor. There is a canonical map $\varpi$ from $\Kset(M) \oplus \Kset(M)$ to the relative homology group $H_1(M,\Sigma^*, \Kset)$ defined as follows: for each $s \in H \backslash G$

\begin{itemize}
\item  $\varpi(\sigma_s)$ is the homology class associated to the lower side of the square $s$, going rightwards;
\item $\varpi(\zeta_s)$ is the homology class associated to the left side of the square $s$,  going upwards.
\end{itemize}
We write $\Box_s$ for the element
$$\Box_s = \sigma_s + \zeta_{sg_r} - \sigma_{sg_u} -\zeta_s,$$
of $\Kset(M) \oplus \Kset(M)$.
Its image in $H_1(M,\Sigma^*, \Kset)$ is equal to $0$. Moreover, we have
$$\sum_s \Box_s  = 0.$$
We denote by $\iota$ the linear map from $\Kset(M)$ to $\Kset(M) \oplus \Kset(M)$ sending $e_s$ to $\Box_s$ for all $s \in
H \backslash G$.
Observe that $\Kset(M)$, $\Kset(M) \oplus \Kset(M)$ and $H_1(M,\Sigma^*, \Kset)$ are naturally equipped with the structure of left ${\rm Aut}(M)$-modules and the maps $\iota$, $\varpi$ are morphisms for these structures. We also equip $\Kset$ with the structure of a trivial left ${\rm Aut}(M)$-module and define $\epsilon: \Kset \rightarrow \Kset(M)$ to be the morphism sending $1$ to $\sum_s e_s$.

\smallskip
We obtain thus a resolution of the relative homology group:

\begin{proposition}
The following is an exact sequence of  left ${\rm Aut}(M)$-modules:
\begin{equation}
0 \rightarrow \Kset \xrightarrow{\epsilon} \Kset(M)\xrightarrow{\iota} \Kset(M) \oplus \Kset(M) \xrightarrow{\varpi} H_1(M,\Sigma^*, \Kset) \rightarrow 0.
\end{equation}

\end{proposition}
\begin{proof}
Denote by ${\rm Sk}(M)$ the complement in $M$ of the union of the (open) squares of $M$. The exact sequence above is nothing else than the exact homology sequence associated to the triple $\Sigma^* \subset {\rm Sk}(M) \subset M$: we have $H_2({\rm Sk}(M),\Sigma^*,\Kset)= H_1(M,{\rm Sk}(M),\Kset)=0$, and the homology groups $H_2(M,\Sigma^*,\Kset)$, $H_2(M,{\rm Sk}(M),\Kset)$,  $H_1({\rm Sk}(M),\Sigma^*,\Kset)$ identify respectively with $\Kset$, $\Kset(M)$ and $\Kset(M) \oplus \Kset(M)$.
\end{proof}

\begin{corollary}
As a ${\rm Aut}(M)$-module, the relative homology group $H_1(M,\Sigma^*, \Kset)$ is isomorphic to the direct sum of $\Kset(M) \oplus \Kset$.
\end{corollary}

Recall the homology exact sequence of left ${\rm Aut}(M)$-modules:

\begin{equation}
0 \rightarrow H_1(M, \Kset) \rightarrow H_1(M,\Sigma^*, \Kset) \rightarrow H_0(\Sigma^*, \Kset)\rightarrow \Kset \rightarrow  0.
\end{equation}

This will be used to determine the structure of left ${\rm Aut}(M)$-module of the absolute homology group $H_1(M, \Kset)$:
the relative  homology group $H_1(M,\Sigma^*, \Kset)$ is given by the corollary above, and  we will study in the next section the left ${\rm Aut}(M)$-module
$H_0(\Sigma^*, \Kset)$.

\medskip

On the other hand, one has a direct sum decomposition into ${\rm Aut}(M)$-submodules (see \cite{MY})
\begin{equation}
 H_1 (M, \Kset) = H_1^{st} (M, \Kset) \oplus H_1^{(0)} (M,\Kset).
\end{equation}
Here, $H_1^{st} (M, \Kset)$ is the $2$-dimensional subspace (with trivial ${\rm Aut}(M)$-action) generated by $\sum_s \sigma_s$ and $\sum_s \zeta_s$, while $H_1^{(0)} (M,\Kset)$ is the kernel of the map $\pi^*: H_1 (M, \Kset) \rightarrow H_1(\Tset^2,\Kset)$.
\smallskip

Observe that the two copies of $\Kset$ in $H_1 (M, \Kset)$ coming from the exact sequences (2.1) and (2.2) correspond to
$H_1^{st} (M, \Kset)$, so we have (as left ${\rm Aut}(M)$-modules)

\begin{equation}
H_1^{(0)} (M, \Kset) = \Kset(M) \ominus H_0(\Sigma^*, \Kset).
\end{equation}

\section{Structure of the ${\rm Aut}(M)$-module $H_1(M, \Kset)$}

\subsection{Action of $N$ on $\Sigma^*$}

The set $\Sigma^* = p^{-1}(0)$ can be identified with the set of orbits for the action of  $<c>$  on $H \backslash G$,
or equivalently to the set of orbits of the action on $G$ of the product group $H \times <c>$ given by $(h,c^m,g) \mapsto hgc^{-m}$. We will denote by $A_g$ the point of $\Sigma^*$ corresponding to the orbit of $g$ under this action, and by
${\rm Stab}(g) \subset N$ the stabilizer of $A_g$ for the left action of $N$ on $\Sigma^*$. We have
\begin{equation}
{\rm Stab}(g) = N \cap H<gcg^{-1}> = H(N \cap <gcg^{-1}>).
\end{equation}
We will denote by $n(g)$ (resp. $h(g)$) the least positive integer such that $gc^{n(g)}g^{-1} $ belongs to $N$ (resp. $H$). Then $n(g)$ is a divisor of $h(g)$ and $h(g)$ is a divisor of the order $\kappa$ of $c$ in $G$. Moreover
 $N \cap <gcg^{-1}>$ is the cyclic group of order $\frac{\kappa}{n(g)}$ generated by $gc^{n(g)}g^{-1} $ and ${\rm Stab}(g) /H$ is cyclic of order $\frac{h(g)}{n(g)}$. We thus obtain:

\begin{proposition}
The size of the orbit $\Sigma_g$ of $A_g$ under the action of $N$ on $\Sigma^*$ is $\frac{n(g)}{h(g)}\# (N/H) $. The number of elements $g'$ such that $A_{g'}$ belongs to $\Sigma_g$  is $n(g)\# N $.
\end{proposition}

The permutation representation of ${\rm Aut}(M)=N/H$ on $\Kset ^{\Sigma_g}$ is induced by the trivial representation of  ${\rm Stab}(g)/H$. The character $\chi_g$ of this representation is given, for $n \in N$ with image $\bar n \in N/H$, by
\begin{eqnarray*}
 \chi_{g} (\bar n) &=& (\# {\rm Stab}(g)/H)^{-1} \sum_{\bar\nu\in N/H} {\bf 1}_{{\rm Stab}(g)/H}(\bar \nu \bar n \bar \nu^{-1})\\
                   &=& \frac{n(g)}{h(g)} (\#H)^{-1}\sum_{\nu\in N} {\bf 1}_{{\rm Stab}(g)}(\nu  n \nu^{-1}).
  \end{eqnarray*}
Thus, the character $\chi_{\Sigma^*}$ of the  representation of ${\rm Aut}(M)=N/H$ on $H_0(\Sigma^*, \Kset)$ is given by
\begin{eqnarray*}
\chi_{\Sigma^*}(\bar n) &=& \sum_{g\in G} (\#N n(g))^{-1} \chi_{g} (\bar n) \\
                       &=& \frac 1{\#N \#H} \sum_{g\in G}  \sum_{\nu\in N} h(g)^{-1} {\bf 1}_{{\rm Stab}(g)}(\nu  n \nu^{-1}).
\end{eqnarray*}

\subsection{Decomposition of $H_0(\Sigma^*, \Cset)$ into isotypic components}

In this subsection, we assume that $\Kset = \Cset$.

\smallskip

Let $\chi_{\alpha}$ be an irreducible character (over $\Cset$) of the group $N/H$. We also write $\chi_{\alpha}$ for the corresponding character of $N$. The multiplicity $m_{\alpha}$ of $\chi_{\alpha}$ in $\chi_{\Sigma^*}$ is given by:

\begin{eqnarray*}
m_{\alpha} &=& (\# N/H)^{-1} \sum_{\bar n\in N/H} \chi_{\alpha}(\bar n) \chi_{\Sigma^*}(\bar n) \\
        &=& (\#N)^{-1} \sum_{n\in N} \chi_{\alpha}(n) \chi_{\Sigma^*}(\bar n) \\
        &=& (\#N)^{-2} (\#H)^{-1} \sum_{\nu\in N} \sum_{g\in G} \sum_{n\in N} h(g)^{-1} \chi_{\alpha}(n){\bf 1}_{{\rm Stab}(g)}(\nu  n \nu^{-1}) \\
        &=& (\#N)^{-2} (\#H)^{-1} \sum_{g\in G} \sum_{s \in {\rm Stab}(g)}\sum_{\nu\in N} \sum_{n\in N} h(g)^{-1} \chi_{\alpha}(\nu n \nu^{-1}) \delta_{s, \nu n \nu^{-1}} \\
        &=& (\#N)^{-2} (\#H)^{-1} \sum_{g\in G} \sum_{s \in {\rm Stab}(g)}\sum_{\nu\in N} h(g)^{-1} \chi_{\alpha}(s) \\
        &=& (\#N)^{-1} (\#H)^{-1} \sum_{g\in G} h(g)^{-1} \sum_{s \in {\rm Stab}(g)}  \chi_{\alpha}(s) .
\end{eqnarray*}

For $g \in G$, we have

\begin{eqnarray*}
(\#H)^{-1} \sum_{s\in {\rm Stab}(g)}  \chi_{\alpha}(s) &=& \sum_{\bar s\in {\rm Stab}(g)/H}  \chi_{\alpha}(\bar s) \\
     &=& \sum_{0 \leq j < \frac {h(g)}{n(g)}} \chi_{\alpha}(gc^{jn(g)}g^{-1}) \\
     &=& \frac {h(g)}{n(g)} \dim Fix_{\alpha}(gc^{n(g)}g^{-1}),
\end{eqnarray*}
where $Fix_{\alpha}(n)$ denotes the subspace of fixed vectors under $n$ for the representation corresponding to $\chi_{\alpha}$. We obtain thus

\begin{proposition} The multiplicity  of $\chi_{\alpha}$ in $\chi_{\Sigma^*}$ is

\begin{equation}
 m_{\alpha} = (\#N)^{-1} \sum_{G} n(g)^{-1} \dim Fix_{\alpha}(gc^{n(g)}g^{-1}).
 \end{equation}
\end{proposition}

\subsection{Decomposition of $H_1^{(0)}(M, \Cset)$ into isotypic components}

Let $\ell_{\alpha}$ be the multiplicity of $\chi_{\alpha}$ in the character of the $N/H$-module $H_1^{(0)} (M,\Cset)$. The  representation of $N/H$ associated to $\Cset(M)$ is  the direct sum of $[G:N]$ copies of the regular representation. From formula (2.4), we have

$$ \ell_{\alpha} = \frac {\#G}{\#N} \dim (\chi_{\alpha}) - m_{\alpha}.$$

With the formula for $m_{\alpha}$ above, this gives

\begin{eqnarray*}
\ell_{\alpha} &=& \frac {\#G}{\#N}\dim (\chi_{\alpha})- \frac 1{\#N} \sum_G n(g)^{-1} \dim Fix_{\alpha}(gc^{n(g)}g^{-1})\\
           &=& \frac 1{\#N} \sum_G (\dim (\chi_{\alpha}) - n(g)^{-1} \dim Fix_{\alpha}(gc^{n(g)}g^{-1})).
\end{eqnarray*}
Observe that the term in the sum is constant in each class $Ng$ of $G$. Choose representatives $g_1=1, \ldots ,g_r$ of these classes (with $r= \frac {\#G}{\#N}$). We have

\begin{eqnarray*}
\ell_{\alpha} &=& \sum_1^r (\dim(\chi_{\alpha}) - n(g_i)^{-1} \dim Fix_{\alpha}(g_ic^{n(g_i)}g_i^{-1})).
\end{eqnarray*}

The representation associated to $\chi_{\alpha}$ defines a representation of $N$ (by composition with the homomorphism $N \rightarrow N/H$). Let $\psi_{\alpha}$ be the representation of $G$ induced by this representation of $N$. We have

$$\dim \psi_{\alpha} = r \dim \chi_{\alpha}.$$

Write the total space $F_{\alpha}$ of the representation of $\psi_{\alpha}$ as

$$ F_{\alpha} = \oplus_1^r g_i^{-1}. E_{\alpha},$$
where  $E_{\alpha}$ is the total space of the representation of $\chi_{\alpha}$.

Let $\rho$ be the permutation of $\{1,\ldots,r\}$ corresponding to the action of $c$ on $G/N$, i.e., satisfying

$$cg_i^{-1} \in  g_{\rho(i)}^{-1}N.$$

We have then

$$\psi_{\alpha}(c) (g_i^{-1}. E_{\alpha}) = g_{\rho(i)}^{-1}. E_{\alpha}.$$

Thus, for any $1 \leq i \leq r$, the integer $n(g_i)$ is the length of the cycle of $\rho$ containing $i$ and is in particular the same for all elements of this cycle. Observe also that the action of $\psi_{\alpha}(c^{n(g_i)})$
on the invariant subspace $g_i^{-1}. E_{\alpha}$ is conjugate to the action of $\chi_{\alpha}(g_ic^{n(g_i)}g_i^{-1})$
on $E_{\alpha}$. The sum over a cycle $C$ of $\rho$, of length $n_C$, of the terms
$$n(g_i)^{-1} \dim Fix_{\alpha}(g_ic^{n(g_i)}g_i^{-1})$$
is therefore equal to the dimension $f_{\alpha}(C)$ of the fixed space of $\psi_{\alpha}(c^{n_C})$ in any $g_i^{-1}. E_{\alpha}$, $i\in C$. The final formula for $\ell_{\alpha}$ is thus

\begin{theorem}\label{t.3-3}
The multiplicity $\ell_{\alpha}$ of $\chi_{\alpha}$ in the character of the $N/H$-module $H_1^{(0)} (M,\Cset)$ is given by
\begin{equation}
\ell_{\alpha} = \dim \psi_{\alpha} - \sum_{C \,{\rm cycle \;of } \;\rho} f_{\alpha}(C).
\end{equation}
\end{theorem}

\begin{corollary}\label{c.3-4}
The multiplicity $\ell_{\alpha}$ is at least equal to the multiplicity $\ell_0$ of the trivial character.
\end{corollary}

\begin{proof}
Indeed, one has
\begin{eqnarray*}
\ell_{\alpha} & = &  \sum_{C \,{\rm cycle \;of } \;\rho} (n_C \dim \chi_{\alpha}- f_{\alpha}(C))\\
              & \geq &  \sum_{C \,{\rm cycle \;of } \;\rho}  (n_C -1)\dim \chi_{\alpha} \\
              & \geq  &   \sum_{C \,{\rm cycle \;of } \;\rho} (n_C -1) \\
              & = & \ell _0.
\end{eqnarray*}
\end{proof}

\subsection{The case of regular origami}

Assume that $\pi: M \rightarrow \Tset^2$ is a regular origami. In this case, we have $H= \{1\}$ and $G=N$ in the discussion above.
Consequently, we have $n(g)=1$ for all $g \in G$. This gives

\begin{corollary}\label{c.3-5}
In the regular case, the multiplicities of $\chi_{\alpha}$ in $H_0(\Sigma^*, \Cset)$ and $H_1^{(0)} (M,\Cset)$ are given by
\begin{equation}
m_{\alpha} = \dim Fix_{\alpha}(c), \quad \quad \ell_{\alpha} ={\rm codim}\, Fix_{\alpha}(c).
\end{equation}
\end{corollary}

\begin{corollary}\label{c.3-6}
When $\dim \chi_{\alpha}=1$, one has $m_{\alpha}=1$, $\ell_{\alpha}=0$.
\end{corollary}
\begin{proof}
Indeed, $c$ is a commutator, and therefore belongs to the kernel of any homomorphism from $G$ into $\Cset^*$.
\end{proof}

Conversely,

\begin{corollary}\label{c.3-7}
When $\dim \chi_{\alpha}>1$, one has  $\ell_{\alpha}>1$.
\end{corollary}
Thus the multiplicity $\ell_{\alpha}$ in the regular case is never equal to $1$. We will see below that this is actually true for all origamis, regular or not!

\begin{proof}
First we show that $\ell_{\alpha}>0$. If we had $\ell_{\alpha}=0$, this would mean that $c$ belongs to the kernel of the irreducible representation $\pi_{\alpha}$ of $G$ with character $\chi_{\alpha}$ of $\textrm{dim}\chi_{\alpha}>1$. But $G$ is generated by $g_r,g_u$ which have $c$ as commutator, hence $\pi_{\alpha}$ would factor through an abelian group. As $\pi_{\alpha}$ is irreducible, this is only possible when $\dim \chi_{\alpha}=1$.

\smallskip
Assume now by contradiction that $\ell_{\alpha}=1$. Thus $\pi_{\alpha}(c)$ fixes a hyperplane in the space $E_{\alpha}$
 of $\pi_{\alpha}$. But, as $c$ is a commutator, one has $\det \pi_{\alpha}(c) =1$. Thus $1$ is the only eigenvalue of $\pi_{\alpha}(c)$. As $\pi_{\alpha}(c)$ is of finite order, this can only happen when $\pi_{\alpha}(c)=1$, a case which was excluded above.
\end{proof}

\subsection{Quasiregular origamis}

As a special case of Corollary \ref{c.3-6}, we see that the trivial representation is never a factor of $H_1^{(0)} (M,\Cset)$ when $M$ is a regular origami.

\begin{definition}
An origami $\pi: M \rightarrow \Tset^2$ is {\it quasiregular} if the trivial representation is not a factor of $H_1^{(0)} (M,\Cset)$, i.e., $\ell_0=0$.
\end{definition}

\begin{proposition}\label{p.3-9}
The origami $M$ associated to the data $G,H,g_r,g_u$ is quasiregular iff the commutator $c$ is contained in the intersection of the conjugates of $N$, or, equivalently, the normal subgroup generated by $c$ is contained in $N$.
\end{proposition}

\begin{proof}
Indeed, from Theorem \ref{t.3-3}, we see that $\ell_0 =0$ iff $n(g) =1$ , i.e., $gcg^{-1} \in N$, for all $g \in G$.
\end{proof}

\begin{corollary}\label{c.3-10}
Assume that the origami $M$ is not quasiregular. Then all irreducible representations (over $\Cset$) of ${\rm Aut}(M)$ have multiplicities $>1$ in $H_1^{(0)} (M,\Cset)$.
\end{corollary}

\begin{proof}
By Corollary \ref{c.3-4}, it is sufficient to show that the multiplicity $\ell_0$ of the trivial representation is $>1$. Let $\rho$ be the permutation  of $G/N$ associated to the action of $c$. From Theorem \ref{t.3-3}, $\ell_0$ is equal to the number of elements of $G/N$ minus the number of cycles of $\rho$.   Since $c$ is a commutator, $\rho$ is an even permutation, and, in particular, $\rho$ is not a transposition.  Also, as $M$ is not quasiregular, $\rho$ is not the identity. This shows that $\ell_0$ is at least equal to $2$.
\end{proof}

\begin{example}
Let $p$ be a prime number and let $G$ be the Heisenberg group over the field $\Fset_p$: the elements of $G$ are the matrices
$$ M(a,b,c)= \left ( \begin{array}{ccc} 1 & a & c \\ 0 & 1 & b \\ 0 & 0 & 1 \end{array} \right ),$$
with $a,b,c \in \Fset_p$. We take $g_r = M(1,0,0)$, $g_u = M(0,1,0)$, $H= <g_u>$. Then the normalizer $N$ is formed of matrices $M(a,b,c)$ with $a=0$ and $N/H $ is isomorphic to $\Fset_p$. As $N$ is a normal subgroup containing  the commutator of $g_r, g_u$, the origami $M$ defined by these data is quasiregular but not regular.
\end{example}

In Subsection \ref{ss.qro} we will give  more sophisticated examples where the automorphism group is not abelian.

\begin{proposition}\label{p.qr}
The origami $M$ associated to the data $G,H,g_r,g_u$ is quasiregular iff the normalizer $N$ is a normal subgroup of $G$ and the quotient group $G/N$ is abelian.
\end{proposition}

\begin{remark}\label{r.qro-ab-nc}
When $M$ is quasiregular, the quotient $G/N$ is generated by the images of $g_r$, $g_u$, therefore it is either cyclic or the product of two cyclic groups. In Subsection \ref{ss.qro-ab-nc}, we give an example where $G/N$ is not cyclic.
\end{remark}
\begin{proof}
If $N$ is normal with $G/N$ abelian, then the commutator $c$ belongs to $N$ and $M$ is quasiregular according to Proposition \ref{p.3-9}. Conversely, assume that $M$ is quasiregular. Let $N_0$ be the intersection of the conjugates of $N$. Then $N_0$ is a normal subgroup of $G$ and $G/N_0$, which is generated by the images of $g_r$, $g_u$, is abelian since the commutator $c$ belongs to $N_0$ (by Proposition \ref{p.3-9}). But then $N/N_0$ is a normal subgroup of $G/N_0$, hence $N$ is normal in $G$. As the images of $g_r$, $g_u$  in $G/N$  span $G/N$ and commute, $G/N$ is abelian.
\end{proof}

\begin{proposition}\label{p.3-14}
Let $M$ be a quasiregular origami. The multiplicity in $H_1^{(0)} (M,\Cset)$ of any irreducible representation (over $\Cset$) of ${\rm Aut}(M)$ is never equal to $1$.
\end{proposition}
\begin{proof}
Let $k_r$, $k_u$ be the respective orders of the images of $g_r$, $g_u$ in the abelian group $G/N$. Denote by $n_r, n_u$ the elements $g_r^{k_r}$, $g_u^{k_u}$ of $N$. For $0 < i \leq k_r$, $0 < j \leq k_u$, define $c_{i,j}:= g_r^{i} g_u^j g_r^{-i} g_u^{-j}$. One has $c_{1,1} = c$ and $c_{k_r,k_u} = n_r n_u n_r^{-1} n_u^{-1}$. For any $0 < i \leq k_r$, $0 < j \leq k_u$, we have
\begin{eqnarray*}
 c_{1,j} &=&  c_{1,1} (g_u c_{1,1} g_u^{-1}) \ldots (g_u^{j-1} c_{1,1} g_u^{1-j} ) \\
 c_{i,j} &=&  (g_r^{i-1} c_{1,j} g_r^{1-i} ) \ldots (g_r c_{1,j} g_r^{-1}) c_{1,j}
 \end{eqnarray*}
 Thus $c_{k_r,k_u}$ is the product (in the appropriate order) of the conjugates $g_r^i g_u^j c g_u^{-j}g_r^{-i}$ for $0 \leq i < k_r, \, 0 \leq j < k_u$. By Theorem \ref{t.3-3}, the multiplicity $\ell_{\alpha}$ in $H_1^{(0)} (M,\Cset)$
 of an irreducible representation $\pi_{\alpha}$ of ${\rm Aut}(M)$   is the sum of the codimensions of the fixed spaces of the $g_r^i g_u^j c g_u^{-j}g_r^{-i}$.

\smallskip

We can now argue as in the regular case. The determinant of $\pi_{\alpha} (c_{k_r,k_u})$ is equal to $1$ since $c_{k_r,k_u}$ is a commutator in $N$. Then either all determinants of the  $\pi_{\alpha} (g_r^i g_u^j c g_u^{-j}g_r^{-i})$ are equal to $1$  or at least
 two are distinct from $1$. In the second case, at least two of the codimensions of the fixed spaces are $>0$ and $\ell_{\alpha} >1$. In the first case, if some  $\pi_{\alpha} (g_r^i g_u^j c g_u^{-j}g_r^{-i})$ is distinct from the identity, the
 corresponding fixed space has codimension at least $2$. Otherwise $\ell_{\alpha}=0$.
\end{proof}

As a direct consequence of Corollary \ref{c.3-10} and Proposition \ref{p.3-14}, we get the following statement (announced right after Corollary \ref{c.3-7})

\begin{corollary} For any origami $M$ (quasiregular or not), the multiplicity $\ell_{\alpha}$ is never equal to $1$.
\end{corollary}

\subsection{Decomposition of $H_1^{(0)}(M, \Kset)$ for other subfields of $\Cset$}

One has $H_1^{(0)}(M, \Cset) = H_1^{(0)}(M, \Kset) \otimes \Cset$. The decompositions of  $H_1^{(0)}(M, \Kset)$ and $H_1^{(0)}(M, \Cset)$ into isotypic components are thus related by the following standard facts of the theory of representations
(or semisimple algebras) that we recall for the convenience of the reader (see \cite{S} and references therein for more details). We denote by ${\mathcal Irr}_{\Kset}(\Gamma)$ the set of isomorphism classes of irreducible representations over $\Kset$ of a finite
 group $\Gamma$, by $\chi_{\alpha}$ the character of the representation $\alpha \in {\mathcal Irr}_{\Kset}(\Gamma)$. We denote by $\overline \Kset $ the algebraic closure of $\Kset$ in $\Cset$.

\begin{itemize}
\item A theorem of Brauer (c.f. Theorem 24 in \cite{S}) asserts that ${\mathcal Irr}_{\Kset}(\Gamma) = {\mathcal Irr}_{\Cset}(\Gamma)$ as soon as $\Kset$ contains the $m^{th}$ roots of unity for any $m$ which is the order of some element of $\Gamma$. In particular we always have ${\mathcal Irr}_{\overline \Kset}(\Gamma) = {\mathcal Irr}_{\Cset}(\Gamma)$.
\item The Galois group ${\rm Gal}(\overline\Kset/\Kset)$ acts in a natural way on ${\mathcal Irr}_{\overline \Kset}(\Gamma)$. The elements of ${\mathcal Irr}_{\Kset}(\Gamma)$ are in canonical one-to-one correspondence with the
orbits of this action.  If $a$ is such an orbit, one has
$$\chi_a = m_a \sum_{\alpha \in a} \chi_{\alpha} ,$$
where the integer $m_a \geq 1$ is the {\it Schur index} of $a$.

\smallskip
For the multiplicities in the decompositions of $ H_1^{(0)}(M, \Kset)$, $H_1^{(0)}(M, \Cset)$  into isotypic components (with $\Gamma = {\rm Aut}(M)$), it means that the multiplicities $\ell_{\alpha}, \alpha \in a$ (in the decomposition of $ H_1^{(0)}(M, \Cset)$) are equal and related to the multiplicity $\ell_a$  (in the decomposition of $ H_1^{(0)}(M, \Kset)$) through $\ell_{\alpha} = m_a \ell_a$.

\item Let $V_a$ be an irreducible $\Kset(\Gamma)$-module, associated to an orbit $a$ as above. The commuting algebra of $\Kset(\Gamma)$ in ${\rm End}_{\Kset}(V_a)$ is a skew-field $D_a$. The degree of $D_a$ over its center  is the square $m_a^2$ of the Schur index.

\smallskip
For $\Gamma = {\rm Aut}(M)$, by writing $W_a$ for the isotypic component of type $a$ in $H_1^{(0)}(M, \Kset)$ and $\ell_a$ for the corresponding multiplicity, it means that the commuting algebra of $\Kset(\Gamma)$ in
${\rm End}_{\Kset}(W_a)$ is isomorphic (through the choice of an isomorphism  of $\Kset(\Gamma)$-modules $W_a \rightarrow V_a^{\ell_a}$) to the matrix algebra
$M({\ell_a}, D_a)$ over the skew-field $D_a$.
\end{itemize}

\subsection {The case $\Kset = \Rset$}

When $\Kset = \Rset$, the discussion of the last subsection takes a particularly simple form. There are three types of irreducible representations:
\begin{itemize}
\item {\it Real} representations correspond to the case of an orbit $a$ having a single element and $m_a =1$.  One has $D_a \simeq \Rset$.
\item {\it Complex} representations correspond to the case of an orbit $a$ having two distinct elements which are deduced from each other by complex conjugation. One has $D_a \simeq \Cset$ and $m_a = 1$.
\item {\it Quaternionic} representations correspond to the case of an orbit $a$ having a single element and $m_a =2$. One has $D_a \simeq \Hset$, the field of quaternions.
\end{itemize}

Let
$$  H_1^{(0)}(M, \Rset) = \oplus_{a \in {\mathcal Irr}_{\Rset} ({\rm Aut}(M))} W_a$$
be the decomposition into isotypic components. The canonical symplectic form $\{.,.\}$ on $ H_1^{(0)}(M, \Rset)$ is invariant under ${\rm Aut}(M)$, hence it induces an isomorphism between the ${\rm Aut(M)}$-module
 $ H_1^{(0)}(M, \Rset)$ and its dual. This isomorphism leaves invariant each isotypic component $W_a$, which means that these components are orthogonal to each other w.r.t. the symplectic form $\{.,.\}$ and the restriction
of the symplectic form $\{.,.\}|_{W_a}$ to each $W_a$ is non-degenerate, i.e., $\{.,.\}|_{W_a}$ is a symplectic form on $W_a$.

\smallskip

Denote by $Sp(W_a)$ the group of automorphisms of the ${\rm Aut}(M)$-module $W_a$ which preserve the symplectic form. We discuss separately the three types of representations (using a set of notes \cite{C} by Y. Cornulier as source of inspiration) and denote by $V_a$ an irreducible representation of type $a$.

\begin{itemize}
\item Assume first that $a$ is real. Choose a ${\rm Aut}(M)$-invariant scalar product $\langle.,.\rangle$ on $V_a$  (unique up to a scalar).
\begin{proposition}
The multiplicity $\ell_a$ is even. One can choose an isomorphism of ${\rm Aut}(M)$-modules
$\iota: V_a^{\ell_a} \rightarrow  W_a$ such that the symplectic form on $W_a$ is given (with $\ell_a = 2 \ell'_a$) by
$$ \langle \iota(v_1,\ldots, v_{\ell_a}),\iota(v'_1,\ldots, v'_{\ell_a})\rangle = \sum_{j=1}^{\ell'_a}
(\langle v_j ,v'_{j+ \ell'_a}\rangle- \langle v'_j,v_{j+ \ell'_a}\rangle).$$
Any  $A \in Sp(W_a)$ can be written as
$\iota^{-1} \circ A \circ \iota ( v_1,\ldots, v_{\ell_a}) =  (v'_1,\ldots, v'_{\ell_a})$ with
$v'_i = \sum_j a_{i,j}v_j$. The map $A \mapsto (a_{i,j})$ is an isomorphism from $Sp(W_a)$ onto the symplectic group $Sp(\ell_a, \Rset)$.
\end{proposition}
\begin{proof}
Observe first that endomorphisms of the ${\rm Aut}(M)$-module $V_a^{\ell_a}$ are exactly the maps $A(v_1,\ldots, v_{\ell_a}) =  (v'_1,\ldots, v'_{\ell_a})$ such that $v'_i = \sum_j a_{i,j}v_j$ for some matrix $ (a_{i,j}) \in M(\ell_a,\Rset)$.
Choose {\it any} isomorphism $\iota': V_a^{\ell_a} \rightarrow  W_a$. Transferring the symplectic form to $V_a^{\ell_a}$ through $\iota'$, we get an isomorphism from $V_a^{\ell_a}$ to its dual, which is identified to $V_a^{\ell_a}$ by the scalar product on $V_a$ . As the symplectic form is non-degenerate and ${\rm Aut}(M)$-invariant, this isomorphism is $\Rset({\rm Aut}(M))$-linear and thus associated to a matrix in $M(\ell_a, \Rset)$, which is antisymmetric and invertible. This implies that $\ell_a$ is even. Composing $\iota'$ by an appropriate automorphism of the ${\rm Aut}(M)$-module $V_a^{\ell_a}$, we get $\iota:V_a^{\ell_a} \rightarrow  W_a$ such that the symplectic form is given by the required canonical expression. The last assertion is an immediate verification.
\end{proof}
\item  Assume now that $a$ is complex. Choose an isomorphism between $D_a$ and $\Cset$ ; this amounts to choose one of the two elements of $a$, call it $\alpha$. Equip $V_a$ with the structure of $\Cset$-vector space deduced from the action of $D_a$. The character of  $V_a$ as $\Cset({\rm Aut}(M))$-module is $\chi_{\alpha}$. Choose a ${\rm Aut}(M)$-invariant hermitian scalar product $\langle.,.\rangle$ on $V_a$  (unique up to a positive real scalar). The real and the imaginary part of this hermitian scalar product form a basis of the  ${\rm Aut}(M)$-invariant $\Rset$-bilinear forms on $V_a$.
\begin{proposition}
There exist integers $p,q$ with $p+q=\ell_a$ and an isomorphism of ${\rm Aut}(M)$-modules
$\iota: V_a^{\ell_a} \rightarrow  W_a$ such that the symplectic form on $W_a$ is given  by
$$ \{ \iota(v_1,\ldots, v_{\ell_a}),\iota(v'_1,\ldots, v'_{\ell_a})\} = \Im (\sum_{m=1}^{p}
\langle v_m ,v'_{m}\rangle -\sum_{m=p+1}^{p+q} \langle v_m ,v'_{m}\rangle ).$$
Any  $A \in Sp(W_a)$ can be written as
$\iota^{-1} \circ A \circ \iota ( v_1,\ldots, v_{\ell_a}) =  (v'_1,\ldots, v'_{\ell_a})$ with
$v'_m = \sum_n a_{m,n}v_n$, $ a_{m,n} \in \Cset$. The map $A \mapsto (a_{m,n})$ is an isomorphism from $Sp(W_a)$ onto the unitary group $U_{\Cset}(p,q)$ of the hermitian form $\sum_1^p |z_m|^2 -\sum_{p+1}^{p+q} |z_m|^2$.
\end{proposition}

\begin{proof}
Choose some isomorphism of ${\rm Aut}(M)$-modules $\iota': V_a^{\ell_a} \rightarrow  W_a$. Transfer  the symplectic form from $W_a$ to $ V_a^{\ell_a}$ through $\iota'$ and write it as
$$  \{ \iota'(v_1,\ldots, v_{\ell_a}),\iota'(v'_1,\ldots, v'_{\ell_a})\} = \sum_{m,n} b_{m,n}(v_m,v'_n),$$
for $\Rset$-bilinear forms $b_{m,n}$ on $V_a$. As $\{\;,\;\}$ is  ${\rm Aut}(M)$-invariant, the same is true of the $b_{m,n}$. All invariant bilinear forms $b$ on $V_a$ satisfy $b(iv,iv')=b(v,v')$, hence we have (writing $v=(v_1,\ldots,v_{\ell_a}), v'=(v'_1,\ldots,v'_{\ell_a})$)
$$  \{ \iota'(iv),\iota'(iv')\} =  \{ \iota'(v),\iota'(v')\}.$$
Define then
$$ H_{\iota'}(v,v') := \{\iota'(iv),\iota'(v')\} +i \{\iota'(v),\iota'(v')\}.$$
This is a non degenerate hermitian form on $V_a^{\ell_a}$ which is  ${\rm Aut}(M)$-invariant (because the symplectic form is ${\rm Aut}(M)$-invariant and ${\rm Aut}(M)$ acts $\Cset$-linearly). In particular, we can write
$$H_{\iota'}(v,v') = \sum\limits_{m,n}c_{m,n}\langle v_m,v_n'\rangle$$
where $c_{m,n}\in\mathbb{C}$ and $c_{n,m}=\overline{c_{m,n}}$.

It follows that we can compose $\iota'$ with an appropriate automorphism of the ${\rm Aut}(M)$-module $V_a^{\ell_a}$ to get $\iota:V_a^{\ell_a} \rightarrow  W_a$ such that
$$H_{\iota}(v,v'):= \{\iota(iv),\iota(v')\} +i \{\iota(v),\iota(v')\}=\sum_{m=1}^{p}
\langle v_m ,v'_{m}\rangle -\sum_{m=p+1}^{p+q} \langle v_m ,v'_{m}\rangle$$
for some integers $p,q\geq 0$ with $p+q=\ell_a$. This proves the first part of the statement of the proposition. Then, again, the last assertion is an immediate verification.
\end{proof}

\item Consider finally the case where $a$ is quaternionic. We fix an isomorphism between $D_a$ and $\Hset$. We equip $V_a$ with the structure of a right vector space over $\Hset$ by setting
$$ v z = \bar z v,  \quad  z \in \Hset, v \in V_a .$$
Here, $\bar z = a-bi-cj-dk$ is the conjugate of the quaternion $z = a +bi+cj+dk$.

Recall that an hermitian form on a right vector space $V$ over $\Hset$ is a map $H: V \times V \rightarrow \Hset$ which satisfies
$$ H(v, v_1 z_1 + v_2 z_2) = H(v,v_1)z_1 + H(v,v_2)z_2, \quad \forall v,v_1, v_2 \in V ,\, z_1,z_2 \in \Hset,$$
$$ H(v,v') = \overline {H(v',v)},\quad  \forall v,v' \in V.$$
Writing $H=H_0 + H_i i + H_j j + H_k k$, the $\Rset$-bilinear form $H_0$ is symmetric, and the $\Rset$-bilinear forms
$H_i, \, H_j,\, H_k$ are alternate. They are related by
$$ H_0(v,v') = H_i(v,v' i) = H_j (v,v' j) = H_k (v,v' k).$$
For $v,v' \in V$, $z \in \Hset$, we have $H(vz,v'z) = \bar z H(v,v') z$, hence
\begin{eqnarray*}
H_i(v i, v' i) & = & H_i(v, v'),\\
H_i(v j, v' j) & = & -H_i(v, v'),\\
H_i(v k, v' k) & = & -H_i(v, v').
\end{eqnarray*}
Conversely, if a $\Rset$-bilinear alternate form $H_i$ satisfy these relations, one defines an hermitian form by
$$ H(v,v') = H_i(v,v' i) + H_i(v,v') i + H_i(v, v'k) j - H_i(v,v' j) k.$$

On the irreducible ${\rm Aut}(M)$-module $V_a$, there exists, up to a positive real scalar, a unique positive definite
${\rm Aut}(M)$-invariant hermitian form $\langle,\rangle$ (obtained as usual by averaging over the group an arbitrary positive
definite hermitian form). The space of ${\rm Aut}(M)$-invariant $\Rset$-bilinear forms on $V_a$ is $4$-dimensional, generated by the four components of $\langle,\rangle$.

\begin{proposition}\label{p.aut-quaternion}
There exists an isomorphism of ${\rm Aut}(M)$-modules $\iota: V_a^{\ell_a} \rightarrow W_a$ and integers $p,q$ with $p+q= \ell_a$
such that one has, for $v=(v_1,\ldots,v_{\ell_a})$, $ v'=(v'_1,\ldots,v'_{\ell_a}) \in V_a^{\ell_a}$
\begin{eqnarray*}
 \sum_1^p \langle v_m,v'_m\rangle - \sum_{p+1}^{p+q} \langle v_m,v'_m\rangle & = &
 \{\iota(v),\iota(v'i )\} + \{\iota(v),\iota(v' )\} i \\
  &  & + \{\iota(v),\iota(v' k)\} j - \{\iota(v),\iota(v'j )\} k.
  \end{eqnarray*}
An element of $Sp(W_a)$ is of the form
$\iota^{-1} \circ A \circ \iota ( v_1,\ldots, v_{\ell_a}) =  (v'_1,\ldots, v'_{\ell_a})$ with
$v'_m = \sum_n v_na_{n,m}$ , $a_{n,m} \in \Hset$. The map $A \mapsto (a_{m,n})$ is an isomorphism from $Sp(W_a)$ onto the unitary group $U_{\Hset}(p,q)$ of the hermitian form $\sum_1^p \bar z_m z_m - \sum_{p+1}^{p+q} \bar z_m z_m$ over $\Hset ^{\ell_a}$.
\end{proposition}

The proof of this proposition relies on the following three lemmas:

\begin{lemma}
Let $G$ be a finite group, and let $W$ be an isotypic $G$-module of quaternionic type. Let $B$ be an alternate non-degenerate $G$-invariant bilinear form on $W$. Then, there exists a non-zero vector $v \in W$ and $g \in G$ such that
$B(v,g.v) \ne 0$.
\end{lemma}
\begin{proof}
Assume that the conclusion of the lemma does not hold. Then, one has $B(v,g.v')+ B(v',g.v)=0$ for all $v,v' \in W$. As
$B$ is alternate and $G$-invariant, one has $B(v,g^2 .v') = B(v,v')$ for all $v,v' \in W,\, g \in G$. As $B$ is non-degenerate, this implies that $g^2 .v'=v'$ for all $v' \in W, \, g \in G$. Thus $G$ acts through a group where all non trivial elements are of order $2$. Such a group is abelian and $W$ cannot be quaternionic.
\end{proof}

\begin{lemma}\label{l.reducibility-q}
Under the hypotheses of the lemma above, one can write
$$ W = V_1 \oplus \ldots \oplus V_{\ell},$$
where $V_1$, \ldots, $V_{\ell}$ are irreducible $G$-modules which are orthogonal for $B$.
\end{lemma}
\begin{proof}
This is an immediate induction on the multiplicity $\ell$ of $W$.
From the lemma above, one can find $v \in W$ such that the restriction of $B$ to the irreducible $G$-module $V_1$ generated by $v$ is nonzero.
Because $V_1$ is irreducible and $B$ is $G$-invariant, this restriction is non-degenerate.
Then, the $B$-orthogonal $W'$ of $V_1$ in $W$ is $G$-invariant and satisfies $W= V_1 \oplus W'$.
 We conclude by applying to $W'$ the induction hypothesis.
\end{proof}

\begin{lemma}\label{l.normalform-q}
Let $b$ be an alternate ${\rm Aut}(M)$-invariant nonzero $\Rset$-bilinear form on $V_a$.
There exists $u \in  \Hset$ with $\bar u u = 1$ such that the form $b_u(v,v'):= b(vu,v'u)$ satisfies
\begin{eqnarray*}
b_u(v i, v' i) & = & b_u(v, v'),\\
b_u(v j, v' j) & = & - b_u(v, v'),\\
b_u(v k, v' k) & = & - b_u(v, v').
\end{eqnarray*}
\end{lemma}
\begin{proof} Any nonzero alternate ${\rm Aut}(M)$-invariant $\Rset$-bilinear form $b$ on $V_a$ is non-degenerate.
This allows to define an adjoint map $\sigma_b: \Hset \rightarrow \Hset$ through $b(v,v'a)= b(v\sigma_b(a),v')$ (the
$\Rset$-endomorphism $\sigma_b(a)$ of $V_a$ belongs to $\Hset$ as it commutes with the action of ${\rm Aut}(M)$).
The map $\sigma_b$ is a $\Rset$-linear involution satisfying $\sigma_b(a a')= \sigma_b(a') \sigma_b(a)$ and $\sigma_b(a) = a$ for $a \in \Rset$.
 Therefore $\sigma_b$ preserves the set of quaternions $a$ such that $a^2=-1$, which is nothing else than the purely imaginary quaternions of norm $1$.
 Observe that, for the $i$-component $B_i$ of the hermitian scalar product on $V_a$, one has
$$\sigma_{B_i}(i)=-i, \;\sigma_{B_i}(j)=j, \; \sigma_{B_i}(k)=k,$$
and that the properties of $b_u$ required in the statement of the lemma are equivalent to $\sigma_{b_u} = \sigma_{B_i}$.

As any nonzero alternate ${\rm Aut}(M)$-invariant $\Rset$-bilinear form $b$ on $V_a$ is a linear combination of the imaginary components $B_i$, $B_j$, $B_k$ of
$\langle,\rangle$ and thus can be deformed to $B_i$ through nonzero alternate ${\rm Aut}(M)$-invariant $\Rset$-bilinear forms,
we conclude that $\sigma_b$ is the identity on a hyperplane of $\Hset$ containing $\Rset$, and that there exists a unique (up to sign)
quaternion $a_0$ of norm $1$ such that $\sigma_b(a_0)=-a_0$. Moreover, $a_0$ is purely imaginary.

Next we relate, for $\bar u u = 1$,  $\sigma_{b_u}$ to $\sigma_b$. For $v,v' \in V_a$, we have
\begin{eqnarray*}
b_u(v,v'a)& = & b(vu, v'au) \\
          & = & b(vu, v'u\bar u a u) \\
          & = & b(vu\sigma_b(\bar u a u), v'u) \\
          & = & b(vu\sigma_b(\bar u a u)\bar u u, v'u) \\
          & = & b_u(vu\sigma_b(\bar u a u)\bar u, v'),
\end{eqnarray*}
and thus $\sigma_{b_u}(a) = u \sigma_b(\bar u a u)\bar u$. Hence, we have $\sigma_{b_u} = \sigma_{B_i}$ iff  $\bar u i u = \pm a_0$. As $a_0$ is purely imaginary of norm $1$, it is possible to choose $u$ such that $\bar u i u = a_0$, so that $b_u$ has the required properties.
\end{proof}

\begin{proof}[Proof of Proposition \ref{p.aut-quaternion}] According to Lemma \ref{l.reducibility-q} above, one can choose an isomorphism of ${\rm Aut}(M)$-modules
$\iota_0: V_a^{\ell_a} \rightarrow W_a$ such that the symplectic form is written in a diagonal way as
$$ \{\iota_0(v), \iota_0(v')\} = \sum_1^ {\ell_a} \hat{b}^{(m)}(v_m,v'_m),$$
for some alternate nonzero ${\rm Aut}(M)$-invariant $\Rset$-bilinear forms $\hat{b}^{(m)}$ on $V_a$.
According to Lemma \ref{l.normalform-q}, one can find quaternions $u_1, \ldots , u_{\ell_a}$ of norm $1$ such that, setting
$\iota_1(v_1, \ldots , v_{\ell_a}) = \iota_0( u_1 v_1, \ldots, u_{\ell_a} v_{\ell_a})$, one has
 $$ \{\iota_1(v), \iota_1(v')\} = \sum_1^ {\ell_a} b^{(m)}(v_m,v'_m),$$
 with the  alternate nonzero ${\rm Aut}(M)$-invariant $\Rset$-bilinear forms $b^{(m)}$ satisfying, for all $1 \leq m \leq \ell_a$,
  \begin{eqnarray*}
b^{(m)}(v i, v' i) & = & b^{(m)}(v, v'),\\
b^{(m)}(v j, v' j) & = & - b^{(m)}(v, v'),\\
b^{(m)}(v k, v' k) & = & - b^{(m)}(v, v').
\end{eqnarray*}
Set $H_i(v,v'):=\{\iota_1(v),\iota(v')\}$ and
$$ H(v,v') := H_i(v,v'i) + H_i(v,v'j) i + H_i(v,v'k) j - H_i(v,v'j) k.$$
This is a non-degenerate ${\rm Aut}(M)$-invariant hermitian form on $W_a$ which can be written as
 $$ H(v,v') = \sum_1 ^{\ell_a} c_m \langle v_m,v'_m \rangle,$$
 for some nonzero real numbers $c_m$. Composing $\iota_1$ with a real diagonal map and permuting the coordinates
if necessary, we may assume that $c_m =1$ for $1 \leq m \leq p$, $c_m=-1$ for $p+1 \leq m \leq p+q = \ell_a$
(for some integers $p,q$), which proves the first assertion of the proposition. For the second assertion,
any automorphism $A$ of the ${\rm Aut}(M)$-module $W_a$ is of the form
 $\iota^{-1} \circ A \circ \iota ( v_1,\ldots, v_{\ell_a}) =  (v'_1,\ldots, v'_{\ell_a})$ with
$v'_m = \sum_n v_n a_{n,m}$ , $a_{n,m} \in \Hset$. Moreover $A$ preserves the symplectic form on $W_a$
iff  $\iota^{-1} \circ A \circ \iota$ preserves the hermitian form
$$ \sum_1^p \langle v_m,v'_m \rangle - \sum_{p+1}^{p+q} \langle v_m,v'_m \rangle$$
on $V_a^{\ell_a}$. But this last condition is equivalent to the property that the matrix $(a_{m,n})$ belongs to
the unitary group $U_{\Hset}(p,q)$ of the hermitian form $\sum_1^p \bar z_m z_m - \sum_{p+1}^{p+q} \bar z_m z_m$
over $\Hset ^{\ell_a}$.
\end{proof}
\end{itemize}

\section{ The affine group and the Kontsevich-Zorich cocycle}

\subsection{The affine group}

Let $\pi: M \rightarrow \Tset^2$ be an origami. We assume that it is {\it reduced}, i.e one cannot factor
$\pi = p \circ \pi'$ with $\pi':M \rightarrow \Tset^2$ an origami and $p: \Tset^2 \rightarrow \Tset^2$ a
covering of degree $>1$.

\begin{definition}
 The affine group of $M$, denoted by ${\rm Aff}(M)$, is the group of homeomorphisms of $M$ which are lifts of
linear automorphisms of $\Tset^2$.
 \end{definition}

 We identify the group of linear automorphisms of $\Tset^2$ with $GL(2,\Zset)$.
The image of the affine group in $GL(2,\Zset)$ is a subgroup of finite index called the {\it Veech group} of $M$
 and denoted by $GL(M)$. We thus have an exact sequence
 $$1\to\textrm{Aut}(M)\to \textrm{Aff}(M)\to GL(M)\to 1.$$

We will denote by $\textrm{Aff}_+(M)$ the subgroup (of index $1$ or $2$) of $\textrm{Aff}(M)$
formed by the orientation-preserving homeomorphisms. Its image in  $GL(2,\Zset)$ is the intersection
of $GL(M)$ with $SL(2,\Zset)$ and is denoted by $SL(M)$.

\smallskip

Let $\Kset$ be a subfield of $\Cset$. The group ${\rm Aff}(M)$ acts on $H_1(M,\Kset)$, and preserves the decomposition
\begin{equation*}
 H_1 (M, \Kset) = H_1^{st} (M, \Kset) \oplus H_1^{(0)} (M,\Kset).
\end{equation*}

Identifying $ H_1^{st} (M, \Kset)$ with $\Kset^2$, the action on  $ H_1^{st} (M, \Kset)$ is given by the homomorphism
$ \textrm{Aff}(M)\to GL(M)$ composed with the canonical action of $GL(M) \subset GL(2,\Kset)$ on $\Kset^2$.
In order to study the action on $H_1^{(0)} (M,\Kset)$, we restrict to a subgroup of finite index of $ \textrm{Aff}(M)$
in the following way.

\smallskip

The affine group acts by conjugation on the automorphism group $\Gamma := \textrm{Aut}(M)$. We thus have a natural
 homomorphism from $\textrm{Aff}(M)$ into the finite group $\textrm{Aut}(\Gamma)$. We denote by $\textrm{Inn}(\Gamma)$
the (normal) subgroup of inner automorphisms and by $\textrm{Out}(\Gamma)$ the quotient
$\textrm{Aut}(\Gamma) / \textrm{Inn}(\Gamma)$.

\smallskip

\begin{definition}
The intersection of $\textrm{Aff}_+(M)$ with the kernel of the morphism
$$\textrm{Aff}(M) \rightarrow \textrm{Aut}(\Gamma)$$
is called the {\it special restricted affine group} and is denoted by $\textrm{Aff}_{**}(M)$.
\smallskip

The intersection of $\textrm{Aff}_+(M)$ with the kernel of the composition
$$\textrm{Aff}(M) \rightarrow \textrm{Aut}(\Gamma) \rightarrow \textrm{Out}(\Gamma)$$
is called the {\it restricted affine group}
 and is denoted by $\textrm{Aff}_*(M)$. It is equal to the subgroup of  $\textrm{Aff}_+(M)$ generated by
$\Gamma=\textrm{Aut}(M)$ and $\textrm{Aff}_{**}(M)$. Thus the images of  $\textrm{Aff}_{**}(M)$ and $\textrm{Aff}_*(M)$ in
$SL(M)$ are equal. This image is called the {\it restricted Veech group} and denoted by $SL_*(M)$.
\end{definition}

\smallskip

The group $\textrm{Out}(\Gamma)$ acts on the set ${\mathcal Irr}_{\Kset}(\Gamma)$ of isomorphism classes of
irreducible representations over $\Kset$ of $\Gamma$. Composing with the homomorphism
$\textrm{Aff}(M) \to \textrm{Out}(\Gamma)$, we get an action of $\textrm{Aff}(M)$ on
${\mathcal Irr}_{\Kset}(\Gamma)$ with a trivial restriction to $\textrm{Aff}_*(M)$. Let

$$  H_1^{(0)}(M, \Kset) = \oplus_{a \in {\mathcal Irr}_{\Kset} (\Gamma)} W_a$$
be the decomposition into isotypic components.

\begin{proposition} For any $A\in\textrm{Aff}(M)$, $a \in  {\mathcal Irr}_{\Kset} (\Gamma)$, we have $A(W_a) = W_{A.a}$.
\end{proposition}

\begin{proof} Let   $A\in\textrm{Aff}(M)$, $a \in  {\mathcal Irr}_{\Kset} (\Gamma)$. For any  $g\in\Gamma$,
$AgA^{-1}$ preserves $ A(W_a)$; in view of the definitions, $A(W_a)$ is a $\Gamma$-submodule of $W_{A.a}$.
The same argument applied to $A^{-1}$ and $A.a$ shows that $A(W_a) = W_{A.a}$.
\end{proof}

\begin{corollary}
 For any $A\in\textrm{Aff}_*(M)$, $a \in  {\mathcal Irr}_{\Kset} (\Gamma)$, $A$ preserves $W_a$. Moreover, if  $A\in\textrm{Aff}_{**}(M)$ and
$\Kset = \Rset$, the restriction of $A$ to $W_a$ belongs to the group $Sp(W_a)$ .
\end{corollary}

\begin{proof}
Indeed, $\textrm{Aff}_*(M)$ acts trivially on ${\mathcal Irr}_{\Kset} (\Gamma)$, therefore any $A\in\textrm{Aff}_*(M)$ preserves each $W_a$. When $A\in\textrm{Aff}_{**}(M)$,  $A$ commutes with $\Gamma$ and its restriction to $W_a$ is an automorphism of $\Gamma$-module. As $A$ is orientation-preserving, the symplectic form is also preserved.
\end{proof}

\begin{remark}
In Section 5.4 below, we discuss an example where $\textrm{Aff}(M)$ acts in a nontrivial way on ${\mathcal Irr}_{\Rset} (\Gamma)$. However, we do not know an example where $\textrm{Aff}(M)$ acts in a nontrivial way on the isotypic components $W_a$ of $H_1^{(0)} (M,\Rset)$.
\end{remark}


\subsection {Definition of  the Kontsevich-Zorich cocycle in presence of automorphisms}

The definition of the Kontsevich-Zorich cocycle over an origami with a nontrivial automorphism group requires a small adjustment with respect to the usual definition (in, e.g., \cite{Fo}).

\smallskip
Let $\pi: M \rightarrow \Tset^2$ be a reduced origami. Any $g \in GL(2,\Rset)$ induces a diffeomorphism $\bar g$ from $\Tset^2$ to the torus $\Rset^2 / g(\Zset^2)$. When $g \in GL(2,\Zset)$, the composition $\bar g \circ \pi$ is an origami, which is isomorphic to $\pi$ iff $g \in GL(M)$. In this case, an isomorphism from $\bar g \circ \pi$ to $\pi$ is just an affine map of $M$ lifting $\bar g$, but such a map is not unique when the automorphism group
$\Gamma := {\rm Aut} (M)$ is not trivial.

\medskip
One way around this problem is to {\it mark}  origamis in the following way. We choose a point $p^*$ in the fiber $\Sigma^*= \pi^{-1}(0)$, and a  rightward horizontal separatrix $\mathcal G_0$ from $p^*$. The number of such
 separatrices is the ramification index $\kappa$ of $\pi$ at $p^*$. Denote by $\widetilde {SL(2,\Rset)}$ the unique connected group (up to isomorphism) which is a covering of degree $\kappa$ of $SL(2,\Rset)$, and write $g$ for the
 image in $SL(2,\Rset)$ of an element $\tilde g \in \widetilde {SL(2,\Rset)}$ by the canonical projection. There is a unique continuous map $\tilde g \mapsto \mathcal G(\tilde g)$ such that
\begin{enumerate}
\item $\mathcal G({\bf 1})= \mathcal G_0$;
\item for any $\tilde g \in \widetilde {SL(2,\Rset)}$,
 $\mathcal G(\tilde g)$ is a rightward horizontal separatrix from $p^*$ for  $\bar g \circ \pi$.
 \end{enumerate}

\smallskip

 Denote by $\widetilde {SL_0(M)}$ the inverse image of $SL(M)$ in
 $\widetilde {SL(2,\Rset)}$ and by $\widetilde {SL(M)}$ the set of $\tilde g \in \widetilde {SL_0(M)}$ such that
there exists an isomorphism between the origamis $\bar g \circ \pi$ and $\pi$ sending
 $\mathcal G(\tilde g)$ on $\mathcal G_0$.
Observe that such an isomorphism is unique, because an automorphism of a translation surface  with a fixed point which
is not a ramification point must be the identity.

 The subgroup $\widetilde {SL(M)}$ has finite index
in $ \widetilde {SL_0(M)}$. Thus the natural map from $\widetilde {SL(2,\Rset)} / \widetilde {SL(M)}$ to
$SL(2,\Rset) / SL(M) \simeq \widetilde {SL(2,\Rset)} / \widetilde {SL_0(M)}$ is a finite covering,
and there is a canonical isomorphism between  $ \widetilde {SL(M)}$ and a subgroup of finite index of the affine group
$ \textrm{Aff}(M)$.

Recall that  the Teichm\"uller flow on $SL(2,\Rset) / SL(M)$ is  the left multiplication by the one-parameter diagonal subgroup $g^t:= \textrm{diag}(e^t,e^{-t})$ of $SL(2,\Rset)$. There is an unique lift of this flow
 (still called the   Teichm\"uller flow) to
$\widetilde {SL(2,\Rset)} / \widetilde {SL(M)}$.

\begin{definition}
The Kontsevich-Zorich (KZ) cocycle is  the cocycle over the Teichm\"uller flow on
$\widetilde {SL(2,\Rset)} / \widetilde {SL(M)}$ obtained by taking the quotient of  the trivial cocycle on
$\widetilde {SL(2,\Rset)} \times H_1(M,\Rset)$ by the action of $ \widetilde {SL(M)}$ (this action  on the second
factor being defined through the  embedding in $ \textrm{Aff}(M)$).
\end{definition}

\subsection{Lyapunov exponents of the KZ-cocycle}

We recall a special case of a standard result, consequence of the classical Hopf argument (\cite{W}).

\begin{lemma}
Let $X$ be any {\bf connected}  covering space of $SL(2,\Rset)/ SL(M)$ of finite degree. Then the lift to $X$ of the Teichm\"uller flow is ergodic w.r.t. the lift of the Haar measure.
\end{lemma}
\begin{proof}
Indeed, from the hyperbolicity of the Teichm\"uller flow,  any ergodic component for the lift of the Teichm\"uller flow is open.
\end{proof}

This holds in particular for the covering $\widetilde {SL(2,\Rset)} / \widetilde {SL(M)}$ introduced in the last section.
 We  can thus apply  Oseledets multiplicative ergodic theorem to the KZ-cocycle and obtain a.e. constant Lyapunov exponents. The vector bundle of the KZ-cocycle splits into a $2$-dimensional subbundle associated to $H_1^{st} (M, \Rset)$,
giving rise to the extremal exponents $\pm 1$, and a complementary subbundle associated to $ H_1^{(0)} (M,\Rset)$,
which corresponds to the exponents in the open interval $(-1,+1)$.

\smallskip

Although the covering space $\widetilde {SL(2,\Rset)} / \widetilde {SL(M)}$ was needed to {\bf define} the KZ-cocycle, one can {\bf compute} the Lyapunov exponents directly on $SL(2,\Rset)/ SL(M)$. Indeed, choose an open neighborhood $U_0$
of the identity in  $SL(2,\Rset)$ which is disjoint from its right translates by nontrivial elements of $SL(M)$. Let $U$ be the image of $U_0$ in $ SL(2,\Rset)/ SL(M)$. There is a full measure set $Z \subset SL(2,\Rset)/ SL(M)$ such that:
\begin{enumerate}
\item all points in $Z$ are recurrent for the Teichm\"uller flow;
\item any point in $\widetilde {SL(2,\Rset)} / \widetilde {SL(M)}$ above $Z$ is regular for Oseledets theorem.
\end{enumerate}
Let now $x \in U \cap Z$, and let $\wt x \in \widetilde {SL(2,\Rset)} / \widetilde {SL(M)}$ be any point above $x$; as $x$ is recurrent, there exists a sequence $(t_n)_{n\in\Zset}$ going to $\pm\infty$ as $n$ goes to $\pm\infty$ such that $g^{t_n}.x \in U$.
Write $x_0$ for the inverse image of $x$ in $U_0$. We can write in a unique way $ g^{t_n}.x_0 = y_n. h_n$, with $y_n \in U_0$ and $h_n \in SL(M)$. Let $A_n$ be {\bf any} affine map with derivative $h_n$. Let $v \in H_1(M,\Rset)$ be a unit vector associated to the exponent $\theta$ above the orbit of $\wt x$. We have then
\begin{equation}
 \theta = \lim_{n \rightarrow \pm \infty} \frac {\log ||A_n(v)||}{\log ||A_n||}.
 \end{equation}
Observe that, if we choose  a $\Gamma$-invariant norm on $H_1(M,\Rset)$, the right-hand term does not depend on $A_n$ but only on its derivative $h_n$.

\smallskip
On the other hand, writing $ \widetilde {SL_*(M)}$ for the subgroup of $\widetilde {SL(M)}$ formed of the elements
whose  image in $SL(M)$ belongs to $SL_*(M)$, we may also consider the lift of the Teichm\"uller flow and of
the KZ-cocycle to the covering space $\widetilde {SL(2,\Rset)} / \widetilde {SL_*(M)}$. From the lemma, the flow is
still ergodic. The Lyapunov exponents produced by Oseledets theorem at this level are the same than for
$\widetilde {SL(2,\Rset)} / \widetilde {SL(M)}$.

\smallskip

However, over $\widetilde {SL(2,\Rset)} / \widetilde {SL_*(M)}$,
the vector bundle associated to $ H_1^{(0)} (M,\Rset)$ splits into subbundles associated to the isotypic components $W_a$,
$a \in {\mathcal Irr}_{\Rset} (\Gamma)$, each being invariant under the KZ-cocycle.
We can apply  Oseledets theorem to each of these subbundles and conclude that the list of non-extremal Lyapunov
exponents (counted with multiplicity) for the KZ-cocycle splits into sublists, one for each
$a \in {\mathcal Irr}_{\Rset} (\Gamma)$. Observe that the restriction of the KZ-cocycle to each subbundle is still
 symplectic, so the Lyapunov exponents in each sublist are still symmetric with respect to $0$.

\begin{proposition}
Let $A \in \textrm{Aff}(M)$, $a \in {\mathcal Irr}_{\Rset} (\Gamma)$. The Lyapunov exponents
(counted with multiplicity) associated to $W_a$ and $W_{A.a}$ are identical.
\end{proposition}

\begin{proof} Let $U_0$, $U$, $Z$ be as above. Let $h \in SL(M)$ be the derivative of $A$. From the lemma, the lift
of the Teichm\"uller flow to $SL(2,\Rset)/ SL_*(M)$ is ergodic. Therefore, for almost all $x \in U \cap Z$, there
exists $T>0$ such that $g^T.x_0 = y_0.\bar h$ ,
where $x_0$ is the preimage of $x$ in $U_0$,  $y_0 \in U_0$ and $h^{-1}\bar h \in SL_*(M)$. Any
$\bar A \in \textrm{Aff}(M)$ with derivative $\bar h$ acts
on ${\mathcal Irr}_{\Rset} (\Gamma)$ in the same way as $A$, in particular we have $A.a = \bar A.a$.
Let
$\wt x_*, \wt y_*$ be the images of $x_0,y_0$ in $\widetilde {SL(2,\Rset)} / \widetilde {SL_*(M)}$
(we take $U_0$ small enough to lift it as a neighborhood of the identity in $\widetilde {SL(2,\Rset)}$).
Let $v \in W_a$ be a unit vector associated to some exponent $\theta$ above the orbit of $\wt x_*$. Then,
from the formula (4.1) above,  it follows that $\bar A.v \in W_{A.a}$ is a vector
associated to the same exponent $\theta$ above the orbit of $\wt y_*$. This proves that any exponent in
$W_a$ is also an exponent in $W_{A.a}$, and the statement of the proposition follows immediately.
\end{proof}

\subsection{The Lyapunov exponents associated to an isotypic component $W_a$ of $ H_1^{(0)} (M,\Rset)$}

In this section, we fix some isotypic component $W_a$ of $ H_1^{(0)} (M,\Rset)$. We consider the restriction of the
KZ-cocycle to the associated subbundle over $\widetilde {SL(2,\Rset)} / \widetilde {SL_*(M)}$. For a Lyapunov exponent
$\theta$ of this restriction, and an Oseledets  regular point $x \in \widetilde {SL(2,\Rset)} / \widetilde {SL_*(M)}$,
we denote by $W_a(\theta,x)$ the subspace of $W_a$ associated to the exponent $\theta$ over the orbit of $x$.

\begin{proposition}
Every subspace $W_a(\theta,x)$ is invariant under the action of $\Gamma=  \textrm{Aut}(M)$.
\end{proposition}
\begin{proof}
Proceeding as in the last section, one can find a sequence $(A_n)$ in $\textrm{Aff}_{**}(M)$ (depending on $x$!) such that
\begin{equation}
 W_a(\theta,x) -\{0\} = \{v\in W_a - \{0\}, \lim_{n \rightarrow \pm \infty} \frac {\log ||A_n(v)||}{\log ||A_n||} = \theta\}.\end{equation}
In this formula, we can choose a $\Gamma$-invariant norm on $W_a$. Let $g \in \Gamma$.  For all $v \in W_a$, $g \in \Gamma$, we have, as  $(A_n) \in \textrm{Aff}_{**}(M)$
$$ ||A_n \circ g(v) || =  || g \circ A_n(v)|| =  || A_n(v)|| .$$
In view of the characterization of $W_a(\theta,x)$, the proof of the proposition is complete.
\end{proof}

\begin{corollary}
Every subspace $W_a(\theta,x)$ is isomorphic as a  $\Rset(\Gamma)$-module to the direct sum of a finite number of the
irreducible $\Rset(\Gamma)$-module $V_a$. In particular, the multiplicity of every Lyapunov exponent in $W_a$
is a multiple of   $\dim_{\Rset}V_a$.
\end{corollary}

Depending on the type (real, complex, quaternionic) of $a \in {\mathcal Irr}_{\Rset} (\Gamma)$, we may sometimes say more
about the exponents using Corollary 4.4 and the discussion about $Sp(W_a)$ in Section 3.7.

\begin{itemize}
\item When $a$ is {\bf real}, there is nothing to say beyond what is true for general symplectic cocycles: if
$\theta$ is a Lyapunov exponent for $W_a$, then $-\theta$ is also an exponent, with the same multiplicity than $\theta$.
\item Assume that $a$ is {\bf complex} or {\bf quaternionic}.  From Propositions 3.17 and 3.18, there exists nonnegative integers $p,q$ with $p+q = \ell_a$
such that $Sp(W_a)$ is isomorphic to $U_{\Kset}(p,q)$, with $\Kset= \Cset$ or $\Hset$.
\begin{proposition}
The multiplicity of the exponent $0$ in $W_a$ is at least \linebreak $|q-p|\dim_{\Rset} V_a$.
\end{proposition}
\begin{proof}
From Propositions 3.17 and 3.18, there is a nondegenerate symmetric bilinear form $B$ of  signature $(p\dim_{\Rset} V_a,q\dim_{\Rset} V_a)$ which is preserved by the elements of $Sp(W_a)$, and in particular by the elements $A_n$  associated in formula (4.2) to a regular point $x$. Let $W_a^s(x)$ be the subspace of $W_a$ associated to the negative Lyapunov exponents. For $v,w \in W_a^s(x)$, we have
$$  B(v,w) = \lim_{n \rightarrow +\infty} B(A_n.v,A_n.w) =0,$$
hence  $W_a^s(x)$ is an isotropic subspace for $B$. But the maximal dimension of an isotropic subspace for a nondegenerate symmetric bilinear form  of  signature $(P,Q)$ is $\min (P,Q)$. Therefore the dimension of  $W_a^s(x)$ is
$\min (p,q)  \dim_{\Rset} V_a$ at most. Similarly, the dimension of the subspace of $W_a$ associated to the positive Lyapunov exponents is at most $\min (p,q) \dim_{\Rset} V_a$. The proposition follows.
\end{proof}

\end{itemize}

\section{Examples}

The main goal of this section is the illustration of our general discussion  by a few concrete examples of regular and quasiregular origamis. These examples are based on the representation theory of some classical finite groups, for which we refer to \cite{FH}.

\subsection{A family of quasiregular origamis}\label{ss.qro}

Let $n\geq 2$ be an integer, and let $G$ be the group consisting of all permutations of the set $\{1,\dots,2n\}$ which respect the natural partition into even and odd numbers. Let $g_r:=(1,2,\dots,2n)$ and $g_u:=(24)$, so that $c=g_r^{-1} g_u^{-1} g_r g_u = (13)(24)$.

 Let $E_{2n}$ be the set of even numbers in $\{1,\dots,2n\}$, and $O_{2n}$ be  the set of odd numbers. Define $N=\{g\in G: g(E_{2n})=E_{2n}\}$, and $H:=\{g\in N: g|_{E_{2n}} = \textrm{id}|_{E_{2n}}\}$.

\begin{proposition} The following properties hold:
\begin{itemize}
\item $g_r$ and $g_u$ generate $G$;
\item $H\cap g_r H g_r^{-1} = \{\textrm{id}\}$;
\item $N$ is the normalizer of $H$ in $G$ and $N/H\simeq S_n$;
\item $N$ is a normal subgroup of $G$ of index $2$.
\end{itemize}
\end{proposition}

\begin{proof} Observe that $g_r H g_r^{-1}=\{g\in G: g|_{O_{2n}}=\textrm{id}|_{O_{2n}}\}$, so that the second item is clear. Also,  $N$ is a subgroup of index 2 of $G$,  hence $N$ is normal, so that the fourth item follows. For the third item, note that $N$ normalizes $H$, and is a maximal nontrivial subgroup of $G$; as $H$ is not normal by the second item,  $N$ is the normalizer of $H$. Finally, for the first item, one notices that, for $-1 \leq i \leq 2n-4$, $g_r^i g_ug_r^{-i}$ is the transposition $(i+2,i+4)$. An easy induction on $n$ shows that $S_n$ is generated by the transpositions $(1,2),\dots,(n-1,n)$. Applying this separately to $E_{2n}$ and  $O_{2n}$, we obtain that the elements $g_r^i g_u g_r^{-i}, -1 \leq i \leq 2n-4$ generate $N$. Since $g_r\notin N$ and $N$ has index 2 in $G$, we conclude that $g_r$ and  $g_u$ generate $G$.
\end{proof}

This proposition says that the data $(G, H, g_r, g_u)$ determine an origami $M_n$ with automorphism group $Aut(M_n)\simeq N/H\simeq S_n$. Moreover, by Proposition \ref{p.qr}, the fact that $N$ is normal in $G$ and $G/N$ is Abelian implies that $M_n$ is quasiregular.

\begin{remark} For the sake of comparison of $M_n$ with the case of regular origamis associated to the group $S_n$, observe that the image $\overline{c}$ of the commutator $c$ of $g_r,g_u$ in the automorphism group $S_n\simeq N/H$ of $M_n$ is $\overline{c}=(1,2)$: by definition, the image $\overline{c}$ of $c$ is computed by looking at the action of $c$ on the even numbers in $\{1,\dots,2n\}$ (in this case it is just the transposition $(2,4)$) and then renormalizing the even numbers in $\{1,\dots,2n\}$ (by multiplication by $1/2$) in order to get a permutation of $S_n\simeq N/H$. In particular, $\overline{c}$ is not the commutator of a pair of elements of $S_n$ (as such commutators are necessarily even permutations).
\end{remark}

Let $\rho_{\lambda}$ be an irreducible representation of $S_n$. Since $N$ has index 2 in $G$ and $g_r\notin N$, by Theorem \ref{t.3-3}, the multiplicity $\ell_{\lambda}$ of $\rho_{\lambda}$ in $H_1^{(0)}(M_n,\mathbb{C})$ is given by

$$\ell_{\lambda}=\textrm{codim }Fix_{\lambda}(\overline{c})+\textrm{codim }Fix_{\lambda}(g_r\overline{c}g_r^{-1})$$

On the other hand, because
$$g_r c g_r^{-1}=\left\{\begin{array}{cc}(2,4)(1,3) & \textrm{if }n=2\\ (2,4)(3,5) & \textrm{if } n>2\end{array}\right.$$
we have that $g_r \overline{c} g_r^{-1}=(1,2)=\overline{c}$ (after considering the action only on even numbers and renormalizing by multiplication by $1/2$). Thus, the previous formula simplifies to
\begin{equation}\label{e.l-qro}
\ell_{\lambda} = 2\textrm{ codim }Fix_{\lambda}(\overline{c})
\end{equation}

\bigskip

In order to compute the right-hand side of this formula, we will briefly recall some aspects of the (very classical) representation theory of $S_n$ (along the lines of \cite{FH}).

A \emph{Young diagram} is a non-increasing sequence  $\lambda=(\lambda_1,\lambda_2,\dots)$ of nonnegative integers  with $\lambda_i=0$ for large $i$. Writing $|\lambda|:=\sum \lambda_i$, we arrange
$|\lambda|$ boxes  in a left-justified way, the first row consisting of $\lambda_1$ boxes, the second row consisting of $\lambda_2$ boxes, etc.
The {\it dual} Young diagram  $\lambda^*$ of $\lambda$ is obtained by exchanging the lines and columns of $\lambda$: $\lambda^*_i = \sup \{n, \lambda_n \geq i\}$.

A conjugacy class of $S_n$ is given by the order of the  cycles of its elements, so conjugacy classes are associated to Young diagrams  $\lambda$ with $|\lambda|=n$. Such Young diagrams are also associated to the  $\Cset$-irreducible representations  of $S_n$ (see \cite{FH}). All such representations are  defined over $\Qset$, in particular they are real.

\begin{example}\label{ex.1} For the simplest Young diagrams, we have:
\begin{itemize}
\item The list $\lambda=(n)$ corresponds to the \emph{trivial} representation $U$.

\item The dual $(\underbrace{1,\dots,1}_{n}):=(1^n)=(n)^*$ of $(n)$ corresponds to
the \emph{alternating} (\emph{signature}) representation $U'$.

\item The list $\lambda=(n-1,1)$ gives rise to the \emph{standard} representation $V^{n-1}=\{v=(a_1,\dots,a_n): \sum\limits_{j=1}^n a_j=0\}$ (such that $\mathbb{C}^n=U\oplus V^{n-1}$ is the usual \emph{permutation} representation of $S_n$).

\item The dual $(2,1^{n-2})=(n-1,1)^*$ of $(n-1,1)$ is the representation $V\otimes U'$ obtained by taking the tensor product of the standard representation $V$ with the alternating representation $U'$.
\end{itemize}
\end{example}

\begin{remark} More generally, the representation associated to the dual $\lambda^*$ of $\lambda$ can be obtained by taking the tensor product $V_{\lambda}\otimes U'$ of the representation $V_{\lambda}$ corresponding to $\lambda$ with the signature representation $U'$.
\end{remark}

The dimension of the irreducible representation $V_{\lambda}$ can be computed with the aid of the \emph{hook-length formula}:
$$\textrm{dim } V_{\lambda} = \frac{n!}{\prod\limits_{i=1}^n H_i}$$
where $H_i$ is the hook length of the box of number $ i$, that is, the number of boxes to the right in the same row of $i$ plus the number of boxes below in the same column of $i$ plus one (for the box itself).

\bigskip

Coming back to our concrete example of quasi-regular origami $M_n$, since $\overline{c}$ is a transposition,  we get from \eqref{e.l-qro} that
\begin{equation}\label{e.2}
\ell_{\lambda} = \textrm{dim } V_{\lambda} - \chi_{\lambda}(\overline{c}).
\end{equation}

The value  $\chi_{\lambda}(\overline{c})$ is given by \emph{Frobenius formula} (specialized here  to the case of a transposition $\overline{c}$):

\begin{theorem}[Frobenius formula] One has
$$\chi_{\lambda}(\overline{c})=\textrm{dim} V_{\lambda}\cdot \frac{p_2(\lambda)}{|\lambda|(|\lambda|-1)}$$
where $p_2(\lambda):=\sum\limits_{i\geq 1} \lambda_i(\lambda_i-2i+1)$.
\end{theorem}

Observe that $p_2(\lambda^*) = -p_2(\lambda)$. Since $\textrm{dim} V_{\lambda} = \textrm{dim} V_{\lambda^*}$ (cf. Remark 5.4) and $|\lambda|=|\lambda^*|$, we have the following corollary of Frobenius formula (and \eqref{e.2}):

\begin{corollary}\label{c.l-qro} $\ell_{\lambda}+\ell_{\lambda^*} = 2\textrm{ dim } V_{\lambda}$.
\end{corollary}

\begin{example}\label{ex.1'}
\begin{itemize}
\item For $\lambda=(n)$, $\ell_{\lambda}=0$ (this is coherent with the fact that $M_n$ is quasi-regular, i.e., the trivial representation has multiplicity zero in $H_1^{(0)}(M_n,\mathbb{C})$).
 By Corollary \ref{c.l-qro}, for $\lambda=(\underbrace{1,\dots,1}_{n}) = (n)^*$, we have $\ell_{\lambda}=2$.

\item For $n\geq 2$ and $\lambda=(n-1,1)$ we have $\textrm{dim } V_{\lambda} = n-1$, $p_2(\lambda)=n(n-3)$ and thus
$\ell_{\lambda} = 2$. Also, by Corollary \ref{c.l-qro}, for $\lambda=(2,\underbrace{1,\dots,1}_{n-2})=(n-2,2)^*$, we have $\ell_{\lambda}= 2n-4$.

\item For $n\geq 4$ and $\lambda=(n-2,2)$ we have $\textrm{dim } V_{\lambda} = \frac{n(n-3)}{2}$, $p_2(\lambda)=(n-1)(n-4)$ and thus
$\ell_{\lambda} = 2(n-3)$. Also, by Corollary \ref{c.l-qro}, for $\lambda=(2,2,\underbrace{1,\dots,1}_{n-4})=(n-2,2)^*$, we have $\ell_{\lambda}=(n-2)(n-3)$.
\item For $n\geq 3$ and $\lambda=(n-2,1,1)$ we have $\textrm{dim } V_{\lambda}=\frac{(n-1)(n-2)}{2}$, $p_2(\lambda)=n(n-5)$ and thus
$\ell_{\lambda}=2(n-2)$. Also, by Corollary \ref{c.l-qro}, for $\lambda=(3,\underbrace{1,\dots,1}_{n-3})=(n-2,1,1)^*$, we have $\ell_{\lambda}=(n-2)(n-3)$.

\end{itemize}
\end{example}

\begin{remark}The examples above give the multiplicities of all irreducible representations for $n\leq 5$. For $n=6$, it remains
\begin{itemize}
\item $\lambda=(3,3)$, for which $\textrm{dim } V_{\lambda}=5$, $p_2(\lambda)=6$ and $\ell_{\lambda}=4$;
\item $\lambda=(2,2,2)=(3,3)^*$, for which $\textrm{dim } V_{\lambda}=5$, $p_2(\lambda)=-6$ and $\ell_{\lambda}=6$;
\item $\lambda=(3,2,1)=(3,2,1)^*$, for which $\textrm{dim } V_{\lambda}=8$, $p_2(\lambda)=0$ and $\ell_{\lambda}=8$.
\end{itemize}
\end{remark}

\begin{remark} Still concerning regular origamis associated to symmetric groups, we observe that a theorem of O. Ore (cf. \cite{Ore}) says that \emph{every} element $c$ of $A_n$ is the commutator $c=[g_r,g_u]$ of two elements $g_r,g_u$ in $S_n$. However, it is not obvious that, for a given $c$, one can choose $g_r$ and $g_u$ with $c=[g_r,g_u]$ and $(g_r, g_u)$ \emph{generates} $A_n$ or $S_n$. It has been shown in the PhD thesis \cite[Theorem 4.26]{Zm} that there always exists a generating pair $(g_r, g_u)$ of $A_{n}$ with $c=[g_r,g_u]$ when $c$ has a large support (namely, when $c$ moves at least $p+2$ points, where $p$  is a prime such that $\left[\frac{3n}{4}\right]\le p\le n-3$).
\end{remark}

\subsection{A quasi-regular origami $H\backslash G$ with $G/N$ Abelian not cyclic}\label{ss.qro-ab-nc}

As  announced in Remark \ref{r.qro-ab-nc}, this subsection contains the description of a quasiregular origami $H\backslash G$ such that the normalizer $N$ of $H$ is normal in $G$ and $G/N$ is the product of two cyclic groups but is  not cyclic.

The basic idea is to slightly modify the family of examples considered in the previous subsection. More precisely, denote by $E=\{0,1,\dots,7\}$ and consider the natural partition
$$E=E_0\cup E_1\cup E_2\cup E_3=\{0,4\}\cup\{1,5\}\cup\{2,6\}\cup\{3,7\}$$
by residues modulo $4$. Let $g_r:=(01)(45)(2367)$ and $g_u:=(02)(46)(1357)$, and consider
\begin{itemize}
\item $G:=\langle g_r, g_u\rangle$ the group generated by $g_r$ and $g_u$;
\item $H:=\{g\in G: g|_{E_0}=\textrm{id}\}$.
\end{itemize}
Note that $g_r$ and $g_u$ act on the partition $E_0\cup E_1\cup E_2\cup E_3$ by the permutations $\overline{g_r}=(E_0,E_1)(E_2,E_3)$ and $\overline{g_u}=(E_0,E_2)(E_1,E_3)$, so that $G=\langle g_r,g_u\rangle$ acts on $E_0\cup E_1\cup E_2\cup E_3$ by Klein's group $\mathbb{Z}/2\mathbb{Z}\times \mathbb{Z}/2\mathbb{Z}$.

\medskip

Also, observe that $\bigcap\limits_{g\in G}g H g^{-1}=\{\textrm{id}\}$ because $g_rHg_r^{-1}=\{g\in G: g|_{E_1}=\textrm{id}\}$, $g_uHg_u^{-1}=\{g\in G: g|_{E_2}=\textrm{id}\}$ and $g_ug_rHg_r^{-1}g_u^{-1}=\{g\in G: g|_{E_3}=\textrm{id}\}$. Thus, the data $M=(G, H, g_r, g_u)$ defines an origami. Observe that $M$ is \emph{not} a regular origami because $H$ is not trivial: for instance, $g_r^2=(26)(37)\in H$ and $g_u^2=(15)(37)\in H$.

\medskip

Now we consider the normalizer $N$ of $H$ in $G$. Since $G$ acts by Klein's group on $E_0,\dots, E_3$, one has that $nHn^{-1}=H$ if and only if $n(E_i)=E_i$ for all $i=0,\dots,3$. It follows that $N$ is a normal subgroup of $G$ and $G/N$ is isomorphic to Klein's group $\mathbb{Z}/2\mathbb{Z}\times\mathbb{Z}/2\mathbb{Z}\simeq\langle\overline{g_r},\overline{g_u}\rangle$, that is, the origami $M$ has the required properties.

\subsection{A family of regular origamis associated to $SL(2,\Fset_p)$}\label{ss.ro}

Let $p\geq 3$ be an odd prime and $G=SL(2,\mathbb{F}_p)$. The order of $G$ is  $|G|=(p-1)p(p+1)$. The elements
$$g_r=\left(\begin{array}{cc} 1 & a \\ 0 & 1\end{array}\right) \quad \textrm{ and } \quad g_u=\left(\begin{array}{cc} 1 & 0 \\ b & 1\end{array}\right)$$
generate $G$ whenever $\pi:=ab\neq 0$. We will study the regular origami associated to this generating pair.

\medskip
\noindent\textbf{Notations}. In the sequel, $\mathbb{F} = \mathbb{F}_p$, $\varepsilon$ is a generator of the cyclic group $\mathbb{F}^*$;   $\mathbb{F}'=\mathbb{F}_{p^2}$ is a quadratic extension of $\mathbb{F}$ (unique up to isomorphism); $C$ is the cyclic subgroup of order $p+1$ of $(\mathbb{F}')^*$ consisting of elements $x$ with $N(x):=x^{p+1} = 1$,  $\eta$ a generator of $C$. Choosing a basis of $\mathbb{F}'$ over $\mathbb{F}$, we identify $C$ with a subgroup of $SL(2,\mathbb{F})$.

\subsubsection{Conjugacy classes in $SL(2,\mathbb{F}_p)$}

The following table (taken from \cite{FH}) presents the information we will need about the conjugacy classes of $SL(2,\mathbb{F})$, namely, it gives representatives for each ``type'' of class (in the first column), the number of elements on each class of a given ``type'' (in the second column), and the number of classes of a given ``type'' (in the third column).

\bigskip

\begin{tabular}{|c|c|c|}
\hline
\textrm{Representative} & \#\textrm{ of elements in the class} & \#\textrm{ of classes}\\
\hline
$\left(\begin{array}{cc}1 & 0 \\ 0 & 1\end{array}\right)$ & $1$ & $1$ \\
\hline
$\left(\begin{array}{cc}-1 & 0 \\ 0 & -1\end{array}\right)$ & $1$ & $1$ \\
\hline
$\left(\begin{array}{cc}1 & 1 \\ 0 & 1\end{array}\right)$ & $(p^2-1)/2$ & $1$ \\
\hline
$\left(\begin{array}{cc}1 & \varepsilon \\ 0 & 1\end{array}\right)$ & $(p^2-1)/2$ & $1$ \\
\hline
$\left(\begin{array}{cc}-1 & -1 \\ 0 & -1\end{array}\right)$ & $(p^2-1)/2$ & $1$ \\
\hline
$\left(\begin{array}{cc}-1 & -\varepsilon \\ 0 & -1\end{array}\right)$ & $(p^2-1)/2$ & $1$ \\
\hline
$\left(\begin{array}{cc}\varepsilon^j & 0 \\ 0 & \varepsilon^{-j}\end{array}\right), 0<j<\frac{(p-1)}{2}$ & $p(p+1)$ & $(p-3)/2$ \\
\hline
$\eta^j, 0<j<\frac{(p+1)}{2}$ & $p(p-1)$ & $(p-1)/2$ \\
\hline
\end{tabular}

\bigskip

The total number of conjugacy classes is $p+4$.

\subsubsection{Irreducible representations}

Below we list the $p+4$ irreducible representations of $G=SL(2,\mathbb{F})$ (see \cite{FH}).

\begin{itemize}
\item[(a)] the trivial representation $U$.
\item[(b)] the standard representation $V$ coming from the action of $G$ on $\mathbb{P}^1(\mathbb{F}_p)$ (that is, we have a permutation representation of $G$ that we write as $U\oplus V$). The character $\chi_V$ is given in the following table
\begin{center}
\begin{tabular}{|c|c|}
\hline
& $\chi_V$ \\
\hline
$\pm\left(\begin{array}{cc}1 & 0 \\ 0 & 1\end{array}\right)$ & $p$ \\
\hline
$\pm\left(\begin{array}{cc}1 & \ast \\ 0 & 1\end{array}\right)$ & $0$ \\
\hline
$\left(\begin{array}{cc}\varepsilon^j & 0 \\ 0 & \varepsilon^{-j}\end{array}\right)$ & $1$ \\
\hline
$\eta^j$ & $-1$ \\
\hline
\end{tabular}
\end{center}
\item[(c)] let $\tau:\mathbb{F}^*\to\mathbb{C}^*$ be a character with $\textrm{Im}\tau(\varepsilon)>0$ (there are $(p-3)/2$ possible choices of $\tau$), and let
$$B=\left\{g=\left(\begin{array}{cc}a & c \\ 0 & a^{-1}\end{array}\right)\in SL(2,\mathbb{F}_p)\right\}$$
be the usual \emph{Borel subgroup}. Define the character $B\to\mathbb{C^*}, g\mapsto \tau(a)$ and consider the induced representation $W_{\tau}$
of $SL(2,\mathbb{F}_p)$. One has $\textrm{dim} W_\tau = p+1=[G:B]$, and the character $\chi_{W_\tau}$ is given in the following table
\begin{center}
\begin{tabular}{|c|c|}
\hline
& $\chi_{W_\tau}$ \\
\hline
$\left(\begin{array}{cc}1 & 0 \\ 0 & 1\end{array}\right)$ & $p+1$ \\
\hline
$\left(\begin{array}{cc}-1 & 0 \\ 0 & -1\end{array}\right)$ & $(p+1)\tau(-1)$ \\
\hline
$\left(\begin{array}{cc}1 & \ast \\ 0 & 1\end{array}\right)$ & $1$ \\
\hline
$\left(\begin{array}{cc}-1 & \ast \\ 0 & -1\end{array}\right)$ & $\tau(-1)$ \\
\hline
$\left(\begin{array}{cc}\varepsilon^j & 0 \\ 0 & \varepsilon^{-j}\end{array}\right)$ & $\tau(\varepsilon^j) + \tau(\varepsilon^{-j})$ \\
\hline
$\eta^j$ & $0$ \\
\hline
\end{tabular}
\end{center}
\item[(d)] in the case of the character $\tau_*$ with $\tau_*(\varepsilon)=-1$, the induced representation $W_{\tau_*}$ is \emph{reducible}:
indeed, $W_{\tau_*}=W'\oplus W^{''}$ is the sum of two irreducible representations. Observe that $\tau_*(\varepsilon)=-1$ implies that $\tau_*(-1)=(-1)^{(p-1)/2}$. The characters $\chi_{W'}, \chi_{W^{''}}$ are given in the following tables.

For $p\equiv 1 (\textrm{mod } 4)$,

\begin{center}
\begin{tabular}{|c|c|c|}
\hline
& $\chi_{W'}$ & $\chi_{W^{''}}$\\
\hline
$\left(\begin{array}{cc}1 & 0 \\ 0 & 1\end{array}\right)$ & $(p+1)/2$ & $(p+1)/2$ \\
\hline
$\left(\begin{array}{cc}-1 & 0 \\ 0 & -1\end{array}\right)$ & $(p+1)/2$ & $(p+1)/2$ \\
\hline
$\left(\begin{array}{cc}1 & 1 \\ 0 & 1\end{array}\right)$ & $(1+\sqrt{p})/2$ & $(1-\sqrt{p})/2$ \\
\hline
$\left(\begin{array}{cc}1 & \varepsilon \\ 0 & 1\end{array}\right)$ & $(1-\sqrt{p})/2$ & $(1+\sqrt{p})/2$ \\
\hline
$\left(\begin{array}{cc}-1 & -1 \\ 0 & -1\end{array}\right)$ & $(1+\sqrt{p})/2$ & $(1-\sqrt{p})/2$ \\
\hline
$\left(\begin{array}{cc}-1 & -\varepsilon \\ 0 & -1\end{array}\right)$ & $(1-\sqrt{p})/2$ & $(1+\sqrt{p})/2$ \\
\hline
$\left(\begin{array}{cc}\varepsilon^j & 0 \\ 0 & \varepsilon^{-j}\end{array}\right)$ & $(-1)^j$ & $(-1)^j$ \\
\hline
$\eta^j$ & $0$ & $0$ \\
\hline
\end{tabular}
\end{center}

and, for $p\equiv 3 (\textrm{mod } 4)$,

\begin{center}
\begin{tabular}{|c|c|c|}
\hline
& $\chi_{W'}$ & $\chi_{W^{''}}$\\
\hline
$\left(\begin{array}{cc}1 & 0 \\ 0 & 1\end{array}\right)$ & $(p+1)/2$ & $(p+1)/2$ \\
\hline
$\left(\begin{array}{cc}-1 & 0 \\ 0 & -1\end{array}\right)$ & $-(p+1)/2$ & $-(p+1)/2$ \\
\hline
$\left(\begin{array}{cc}1 & 1 \\ 0 & 1\end{array}\right)$ & $(1+i\sqrt{p})/2$ & $(1-i\sqrt{p})/2$ \\
\hline
$\left(\begin{array}{cc}1 & \varepsilon \\ 0 & 1\end{array}\right)$ & $(1-i\sqrt{p})/2$ & $(1+i\sqrt{p})/2$ \\
\hline
$\left(\begin{array}{cc}-1 & -1 \\ 0 & -1\end{array}\right)$ & $(-1-i\sqrt{p})/2$ & $(-1+i\sqrt{p})/2$ \\
\hline
$\left(\begin{array}{cc}-1 & -\varepsilon \\ 0 & -1\end{array}\right)$ & $(-1+i\sqrt{p})/2$ & $(-1-i\sqrt{p})/2$ \\
\hline
$\left(\begin{array}{cc}\varepsilon^j & 0 \\ 0 & \varepsilon^{-j}\end{array}\right)$ & $(-1)^j$ & $(-1)^j$ \\
\hline
$\eta^j$ & $0$ & $0$ \\
\hline
\end{tabular}
\end{center}
\item[(e)] let $\varphi: C\to\mathbb{C}^*$ be a character with $\textrm{Im}\varphi(\eta)>0$ (there are $(p-1)/2$ possible choices of $\varphi$). Using $\varphi$ it is possible to construct a representation $X_{\varphi}$ whose character $\chi_{X_{\varphi}}$ is given by the following table
\begin{center}
\begin{tabular}{|c|c|}
\hline
& $\chi_{X_{\varphi}}$ \\
\hline
$\left(\begin{array}{cc}1 & 0 \\ 0 & 1\end{array}\right)$ & $p-1$ \\
\hline
$\left(\begin{array}{cc}-1 & 0 \\ 0 & -1\end{array}\right)$ & $(p-1)\varphi(-1)$ \\
\hline
$\left(\begin{array}{cc}1 & \ast \\ 0 & 1\end{array}\right)$ & $-1$ \\
\hline
$\left(\begin{array}{cc}-1 & \ast \\ 0 & -1\end{array}\right)$ & $-\varphi(-1)$ \\
\hline
$\left(\begin{array}{cc}\varepsilon^j & 0 \\ 0 & \varepsilon^{-j}\end{array}\right)$ & $0$ \\
\hline
$\eta^j$ & $-(\varphi(\eta^j)+\varphi(\eta^{-j}))$ \\
\hline
\end{tabular}
\end{center}\item[(f)] let $\varphi_*$ be the character with $\varphi_*(\eta)=-1$. In this case, $X_{\varphi_*}$ is \emph{reducible}:
indeed, $X_{\varphi_*} = X'\oplus X^{''}$ is the direct sum of two irreducible representations. Observe that
$\varphi_*(\eta)=-1$ implies that $\varphi_*(-1) = (-1)^{(p+1)/2}$. The characters $\chi_{X'}, \chi_{X^{''}}$ are given in the following tables.

For $p\equiv 1 (\textrm{mod }4)$,

\begin{center}
\begin{tabular}{|c|c|c|}
\hline
& $\chi_{X'}$ & $\chi_{X^{''}}$\\
\hline
$\left(\begin{array}{cc}1 & 0 \\ 0 & 1\end{array}\right)$ & $(p-1)/2$ & $(p-1)/2$ \\
\hline
$\left(\begin{array}{cc}-1 & 0 \\ 0 & -1\end{array}\right)$ & $-(p-1)/2$ & $-(p-1)/2$ \\
\hline
$\left(\begin{array}{cc}1 & 1 \\ 0 & 1\end{array}\right)$ & $(-1+\sqrt{p})/2$ & $(-1-\sqrt{p})/2$ \\
\hline
$\left(\begin{array}{cc}1 & \varepsilon \\ 0 & 1\end{array}\right)$ & $(-1-\sqrt{p})/2$ & $(-1+\sqrt{p})/2$ \\
\hline
$\left(\begin{array}{cc}-1 & -1 \\ 0 & -1\end{array}\right)$ & $(1-\sqrt{p})/2$ & $(1+\sqrt{p})/2$ \\
\hline
$\left(\begin{array}{cc}-1 & -\varepsilon \\ 0 & -1\end{array}\right)$ & $(1+\sqrt{p})/2$ & $(1-\sqrt{p})/2$ \\
\hline
$\left(\begin{array}{cc}\varepsilon^j & 0 \\ 0 & \varepsilon^{-j}\end{array}\right)$ & $0$ & $0$ \\
\hline
$\eta^j$ & $(-1)^{j+1}$ & $(-1)^{j+1}$ \\
\hline
\end{tabular}
\end{center}

and, for $p\equiv 3 (\textrm{mod }4)$,

\begin{center}
\begin{tabular}{|c|c|c|}
\hline
& $\chi_{X'}$ & $\chi_{X^{''}}$\\
\hline
$\left(\begin{array}{cc}1 & 0 \\ 0 & 1\end{array}\right)$ & $(p-1)/2$ & $(p-1)/2$ \\
\hline
$\left(\begin{array}{cc}-1 & 0 \\ 0 & -1\end{array}\right)$ & $(p-1)/2$ & $(p-1)/2$ \\
\hline
$\left(\begin{array}{cc}1 & 1 \\ 0 & 1\end{array}\right)$ & $(-1+i\sqrt{p})/2$ & $(-1-i\sqrt{p})/2$ \\
\hline
$\left(\begin{array}{cc}1 & \varepsilon \\ 0 & 1\end{array}\right)$ & $(-1-i\sqrt{p})/2$ & $(-1+i\sqrt{p})/2$ \\
\hline
$\left(\begin{array}{cc}-1 & -1 \\ 0 & -1\end{array}\right)$ & $(-1+i\sqrt{p})/2$ & $(-1-i\sqrt{p})/2$ \\
\hline
$\left(\begin{array}{cc}-1 & -\varepsilon \\ 0 & -1\end{array}\right)$ & $(-1-i\sqrt{p})/2$ & $(-1+i\sqrt{p})/2$ \\
\hline
$\left(\begin{array}{cc}\varepsilon^j & 0 \\ 0 & \varepsilon^{-j}\end{array}\right)$ & $0$ & $0$ \\
\hline
$\eta^j$ & $(-1)^{j+1}$ & $(-1)^{j+1}$ \\
\hline
\end{tabular}
\end{center}
\end{itemize}

In summary, the $p+4$ irreducible representations are $U$, $V$, $W_\tau$ ($(p-3)/2$ of them), $W',W^{''}$, $X_{\varphi}$ ($(p-1)/2$ of them), $X',X^{''}$.

\medskip

The representations $U$ and $V$ are defined over $\Qset$. To determine whether the other irreducible representations are real, complex or quaternionic, we observe first from the previous tables  that the characters are
real-valued except for $W', W^{''},X',X^{''}$ when $p\equiv 3 (\textrm{mod }4)$. These four representations are thus \emph {complex} (when $p\equiv 3 (\textrm{mod }4)$).
 Observe also that, when $p\equiv 1 (\textrm{mod }4)$, the dimension of $W', W^{''}$ is odd so these representations are \emph {real}.

 Recall the general criterion based on the so-called \emph{Frobenius-Schur indicator}:

\begin{theorem} Let $\chi$ be a character of an irreducible representation of a finite group $G$. Then,
$$\frac{1}{|G|}\sum\limits_{g\in G}\chi(g^2)=\left\{\begin{array}{cl}1, & \textrm{ for } \chi \textrm{ real } \\ 0, & \textrm{ for } \chi \textrm{ complex } \\ -1, & \textrm{ for } \chi \textrm{ quaternionic }\end{array}\right.$$
\end{theorem}

Applying this criterion (with the aid of the previous character tables), one can check that
\begin{itemize}
\item $X',X^{''}$ are \emph{quaternionic} when $p\equiv 1 (\textrm{mod }4)$;
\item for $\tau(\varepsilon) = \exp(2\pi i j/(p-1))$ (with $0<j<(p-1)/2$), $W_{\tau}$ is \emph{real} when $j$ is even, and $W_{\tau}$ is \emph{quaternionic} when $j$ is odd;
\item for $\varphi(\eta)=\exp(2\pi i j/(p+1))$ (with $0<j<(p+1)/2$), $X_{\varphi}$ is \emph{real} when $j$ is even, and $X_{\varphi}$ is \emph{quaternionic} when $j$ is odd.
\end{itemize}
In a nutshell, our discussions so far can be resumed as follows:
\begin{itemize}
\item for $p\equiv 1 (\textrm{mod }4)$, there are $(p+5)/2 = 2+2+(p-3)/2$ real representations, $(p+3)/2 = 2+(p-1)/2$ quaternionic representations, and $0$ complex representations.
\item for $p\equiv 3 (\textrm{mod }4)$, there are $(p+1)/2=2+(p-3)/2$ real representations, $(p-1)/2$ quaternionic representations, and $4=2+2$ complex representations.
\end{itemize}
This concludes our quick review of the representation theory of $G=SL(2,\mathbb{F})$. Now, we pass to the study of the regular origami.

\subsubsection{The regular origami $(G,g_r,g_u)$}

Recall that we have chosen
$$g_r=\left(\begin{array}{cc}1 & a \\ 0 & 1\end{array}\right) \quad \textrm{and} \quad g_u=\left(\begin{array}{cc}1 & 0 \\ b & 1\end{array}\right)$$
with $\pi:=ab\neq 0$. Their commutator $c=g_r^{-1}g_u^{-1}g_r g_u$ is
$$c=\left(\begin{array}{cc}1+ab+a^2b^2 & a^2b \\ -a^2b & 1-ab\end{array}\right)$$
and hence its trace is $\textrm{tr}(c)=2+\pi^2$. The nature of the eigenvalues of $c$ is described by the discriminant
$$(\textrm{tr}(c))^2-4=\pi^2(\pi^2+4)$$

There are three cases:
\begin{itemize}
\item \emph{parabolic}: when $\pi^2=-4$ (this can only happen when $p\equiv 1 (\textrm{mod }4)$ as $-1$ must be a square), then $c$ is conjugated to $\left(\begin{array}{cc}-1 & \ast \\ 0 & -1\end{array}\right)$; the order of $c$ is $2p$.
\item \emph{hyperbolic}: when $\pi^2+4\neq 0$ is a square in $\mathbb F$, $c$ is conjugated to $\left(\begin{array}{cc}\varepsilon^j & 0 \\ 0 & \varepsilon^{-j}\end{array}\right)$ for some $0<j<\frac {p-1}2$. The integer $j$ is \emph {even}: writing $\pi^2+4 = u^2$, $u := \lambda + \lambda^{-1}$ with $\lambda^{\pm 1}:= \frac 12(u \pm \pi) \in \mathbb {F}^*$, we have $\lambda^2 +   \lambda^{-2} = u^2 -2 = \pi^2 +2 = \textrm{tr}(c)$. The order of $c$ is $>2$ and divides $\frac {p-1}2$.
\item \emph{elliptic}: when $\pi^2+4$ is not a square in $\mathbb F$, $c$ is conjugated to $ \eta^j$ for some $0<j<\frac {p+1}2$. The integer $j$ is \emph {odd}: if we had $c = c'^2$, with $\textrm{tr}(c') = \lambda + \lambda^{-1} \in \mathbb F$ for some $\lambda \in \mathbb {F}'$, then we would have $\pi^2+4 = \textrm{tr}(c) +2 = \lambda^2 +   \lambda^{-2}+2 = (\textrm{tr}(c'))^2$. The order of $c$ is $>2$,  even, and divides $p+1$.
\end{itemize}

Recall that (see \cite{Zm}) the genus $g$ of the regular origami $(G,g_r,g_u)$ is related to $\textrm{ord}(c)$ by the  formula:
$$g=1+\frac{1}{2}|G|\left(1-\frac{1}{\textrm{ord}(c)}\right)$$

\begin {remark} Observe that
$$\#\{\pi^2: \pi^2+4\neq 0 \textrm{ is a square} \} =
\left\{\begin{array}{cc}(p-5)/4, & \textrm{ if } p\equiv 1 (\textrm{mod }4) \\ (p-3)/4, & \textrm{ if } p\equiv 3 (\textrm{mod }4)\end{array}\right.$$
and
$$\#\{\pi^2: \pi^2+4 \textrm{ is not a square} \}=\left\{\begin{array}{cc}(p-1)/4, & \textrm{ if } p\equiv 1 (\textrm{mod }4) \\ (p+1)/4, & \textrm{ if } p\equiv 3 (\textrm{mod }4)\end{array}\right.$$

\end{remark}

\begin{remark}

In the hyperbolic or elliptic cases, the order of $c$ can be computed as follows: the sequence $t_n:=\textrm{tr}(c^n)$ satisfies $t_0=2$, $t_1=2+\pi^2$ and the recurrence relation $t_{n+1} = t_1 t_n - t_{n-1}$ (derived for instance from the general formula $\textrm{tr}(A^2B)=\textrm{tr}(A)\cdot\textrm{tr}(AB)-\textrm{tr}(B)$ for $A,B\in SL(2)$). The order of $c$ is the smallest positive integer $n$ such that $t_n =2$.

\end{remark}

For the first values of $p$, one gets the following tables (where  $\left(\begin{array}{c} a \\ p\end{array}\right)$ is the Legendre symbol).

For $p=5$:

\begin{center}
\begin{tabular}{|c|c|c|c|}
\hline
$\pi^2$ & $\pi^2+4$ & $\left(\begin{array}{c} \pi^2+4 \\ p\end{array}\right)$ & $\textrm{ord}(c)$ \\
\hline
$1$ & $0$ & $0$ & $10$ \\
\hline
$4$ & $3$ & $-1$ & $6$ \\
\hline
\end{tabular}
\end{center}

For $p=7$:

\begin{center}
\begin{tabular}{|c|c|c|c|}
\hline
$\pi^2$ & $\pi^2+4$ & $\left(\begin{array}{c} \pi^2+4 \\ p\end{array}\right)$ & $\textrm{ord}(c)$ \\
\hline
$1$ & $5$ & $-1$ & $8$ \\
\hline
$4$ & $1$ & $1$ & $3$ \\
\hline
$2$ & $6$ & $-1$ & $8$ \\
\hline
\end{tabular}
\end{center}

For $p=11$:

\begin{center}
\begin{tabular}{|c|c|c|c|}
\hline
$\pi^2$ & $\pi^2+4$ & $\left(\begin{array}{c} \pi^2+4 \\ p\end{array}\right)$ & $\textrm{ord}(c)$ \\
\hline
$1$ & $5$ & $1$ & $5$ \\
\hline
$4$ & $8$ & $-1$ & $12$ \\
\hline
$9$ & $2$ & $-1$ & $4$ \\
\hline
$5$ & $9$ & $1$ & $5$ \\
\hline
$3$ & $7$ & $-1$ & $12$ \\
\hline
\end{tabular}
\end{center}

For $p=13$:

\begin{center}
\begin{tabular}{|c|c|c|c|}
\hline
$\pi^2$ & $\pi^2+4$ & $\left(\begin{array}{c} \pi^2+4 \\ p\end{array}\right)$ & $\textrm{ord}(c)$ \\
\hline
$1$ & $5$ & $-1$ & $14$ \\
\hline
$4$ & $8$ & $-1$ & $14$ \\
\hline
$9$ & $0$ & $0$ & $26$ \\
\hline
$3$ & $7$ & $-1$ & $14$ \\
\hline
$12$ & $3$ & $1$ & $6$ \\
\hline
$10$ & $1$ & $1$ & $3$ \\
\hline
\end{tabular}
\end{center}

In the next subsections, we use Corollary \ref{c.3-5}, the tables for the characters and the formula

\begin{equation}\label{eq.dimfix}
\textrm{dim }Fix_{\lambda}(c) = \frac 1{\textrm{ord}(c)} \sum\limits_{j=0}^{\textrm{ord}(c)-1}\chi_{\lambda}(c^j)
\end{equation}
to compute the multiplicities in $H_1^{(0)}$ of the irreducible representations of $G$. The computations are straightforward but fastidious and are omitted.

\subsubsection{Multiplicities in the parabolic case}

The parabolic case  occurs only when $p \equiv 1 (\textrm{mod }4)$. We have then

$$\textrm{dim }Fix_{\lambda}(c)=\left\{\begin{array}{cl}1 & \textrm{ for } \lambda = U, V, W', W^{''} \\ 0 & \textrm{ for } \lambda=X_{\varphi}, X', X^{''} \\ 1+\tau(-1)\in\{0,2\} & \textrm{ for } \lambda = W_{\tau}\end{array}\right.$$

From Corollary  \ref{c.3-5}, we then get for the multiplicities

$$\ell_{\lambda}=\textrm{codim}\textrm{Fix}_{\lambda}(c)=\left\{\begin{array}{cl}0 & \textrm{ for } \lambda = U \\
p-1 & \textrm{ for } \lambda=V \\ (p-1)/2 & \textrm{ for } W', W^{''} \\ p-1 & \textrm{ for } \lambda=X_{\varphi} \\
(p-1)/2 & \textrm{ for }\lambda = X',X^{''} \\ p-\tau(-1) & \textrm{ for } \lambda=W_{\tau} \end{array}\right.$$

\subsubsection{Multiplicities in the hyperbolic case}

In the hyperbolic case, the order of $c$ is $>2$ and divides $\frac {p-1}2$. We discuss according to the parity of $\textrm{ord}(c)$, noticing that $\textrm{ord}(c)$ can be even only when $p \equiv 1 (\textrm{mod }4)$.
\begin{enumerate}
\item We first assume that $\textrm{ord}(c)$ is even.
\begin{itemize}
\item for $\lambda = U$, we have $\ell_{\lambda}= 0$.
\item for $\lambda = V$, we have $\ell_{\lambda}= (p-1) (1-\frac 2{\textrm{ord}(c)})$.
\item for $\lambda = W_{\tau}$ with $\tau(c) =1$, we have $\ell_{\lambda}= (p-1) (1-\frac 2{\textrm{ord}(c)})$.
\item for $\lambda = W_{\tau}$ with $\tau(c) \ne 1$, we have $\ell_{\lambda}= 2+(p-1)(1-\frac {1 + \tau(-1)}{\textrm{ord}(c)})$.
\item for $\lambda = W'$ or $W''$, we have $\ell_{\lambda}= \frac {p-1}2 (1-\frac 2{\textrm{ord}(c)})$.
\item for $\lambda = X_{\varphi}$, we have $\ell_{\lambda}= (p-1) (1- \frac {1 + \varphi(-1)}{\textrm{ord}(c)})$.
\item for $\lambda = X'$ or $X''$, we have $\ell_{\lambda}= \frac {p-1}2$.
\end{itemize}
\item We now assume that $\textrm{ord}(c)$ is odd.
\begin{itemize}
\item for $\lambda = U$, we have $\ell_{\lambda}= 0$.
\item for $\lambda = V$, we have $\ell_{\lambda}= (p-1) (1-\frac 1{\textrm{ord}(c)})$.
\item for $\lambda = W_{\tau}$ with $\tau(c) =1$, we have $\ell_{\lambda}= (p-1) (1-\frac 1{\textrm{ord}(c)})$.
\item for $\lambda = W_{\tau}$ with $\tau(c) \ne 1$, we have $\ell_{\lambda}= 2+(p-1)(1-\frac {p-1}{\textrm{ord}(c)})$.
\item for $\lambda = W'$ or $W''$, we have $\ell_{\lambda}= \frac {p-1}2 (1-\frac 1{\textrm{ord}(c)})$.
\item for $\lambda = X_{\varphi}$, we have $\ell_{\lambda}= (p-1) (1- \frac {1 }{\textrm{ord}(c)})$.
\item for $\lambda = X'$ or $X''$, we have $\ell_{\lambda}= \frac {p-1}2 (1-\frac 1{\textrm{ord}(c)})$.
\end{itemize}
\end{enumerate}
\subsubsection{Multiplicities in the elliptic case}
In the elliptic case the order of $c$ is even. The multiplicities are as follows
\begin{itemize}
\item for $\lambda = U$, we have $\ell_{\lambda}= 0$.
\item for $\lambda = V$, we have $\ell_{\lambda}= (p+1) (1-\frac 2{\textrm{ord}(c)})$.
\item for $\lambda = W_{\tau}$, we have $\ell_{\lambda}= (p+1)(1-\frac {1 + \tau(-1)}{\textrm{ord}(c)})$.
\item for $\lambda = W'$ or $W''$, we have
$$\ell_{\lambda}= \left\{\begin{array}{cl} (p+1)(\frac 12 - \frac 1 {\textrm{ord}(c)}) & \textrm{for}\; p \equiv 1 (\textrm{mod }4)  \\
\frac {p+1}2 & \textrm{for} \;p \equiv 3 (\textrm{mod }4)\end{array}\right.$$.
\item for $\lambda = X_{\varphi}$ with $\varphi(c) =1$, we have $\ell_{\lambda}= (p+1) (1-\frac 2{\textrm{ord}(c)})$.
\item for $\lambda = X_{\varphi}$ with $\varphi(c) \ne 1$, we have $\ell_{\lambda}= -2 +(p+1) (1- \frac {1 + \varphi(-1)}{\textrm{ord}(c)})$.
\item for $\lambda = X'$ or $X''$, we have
$$\ell_{\lambda}= \left\{\begin{array}{cl} \frac {p-1}2  & \textrm{for} \;p \equiv 1 (\textrm{mod }4)  \\
 -1+(p+1)(\frac 12 - \frac 1 {\textrm{ord}(c)})& \textrm{for} \; p \equiv 3 (\textrm{mod }4)\end{array}\right.$$
\end{itemize}

\begin{remark}It would be interesting to compute the signatures of the natural Hermitian forms on the isotypic components of complex or quaternionic type.
\end{remark}



\subsection {A regular origami $\mathcal O$ with a non trivial action of the affine group on ${\mathcal Irr}_{\Rset} ({\rm Aut}(\mathcal O))$}

As in example 3.11, we consider a prime number $p$ and    the Heisenberg group $G$ over the field $\Fset:=\Fset_p$:
the elements of $G$ are the matrices
$$ M(a,b,c)= \left ( \begin{array}{ccc} 1 & a & c \\ 0 & 1 & b \\ 0 & 0 & 1 \end{array} \right ),$$
with $a,b,c \in \Fset$. We consider the {\it regular} origami  $\mathcal O$ associated to the choice
$g_r = M(1,0,0)$, $g_u = M(0,1,0)$ of generators of $G$. This origami was studied in \cite{He}, \cite{Zm}.
As for any regular origami, the automorphism group of $\mathcal O$ is canonically
identified with $G$. Its Veech group is equal to $GL(2,\Zset)$ (see \cite{He}, \cite{Zm}).

\smallskip

The automorphism group of $G = \textrm {Aut} (\mathcal O)$ fits into an exact sequence
$$ 1 \rightarrow \textrm{Inn} (G) \simeq \Fset^2 \rightarrow \textrm {Aut} (G) \rightarrow GL(2,\Fset) \simeq \textrm{Out}(G) \rightarrow 1\; .$$
Here, the automorphism of $G$ associated to a pair $(c,c') \in \Fset^2$ sends $g_r$ to $M(1,0,c)$ and
$g_u$ to $M(0,1,c')$. The morphism $r$ from $\textrm {Aut} (G)$ to $GL(2,\Fset)$ is defined as follows: let
$\varphi \in \textrm {Aut} (G)$; write $\varphi(g_r)=M(a,b,c)$, $\varphi(g_u)=M(a',b',c')$; then  $r(\varphi)$  is the matrix
$\left ( \begin{array}{cc} a & a' \\ b & b' \end{array} \right )$.
Observe that
$$ \varphi (M(0,0,1)) = M(0,0, \det r(\varphi))\;.$$

\smallskip

Let us briefly recall what are the irreducible representations of $G$ over $\Cset$. There are $p^2$ $1$-dimensional representations, indexed by $(m,n) \in \Fset^2$, defined by the homomorphisms
$$ \chi_{m,n}(M(a,b,c)) = \exp \frac {2 \pi i }{p} (ma+nb) \;,$$
and $(p-1)$ irreducible representations of dimension $p$, indexed by primitive roots $\zeta$ of order $p$, defined (up to isomorphism) by
\begin{eqnarray*}
\rho_{\zeta} (g_r) (e_i) &=& \zeta^i e_i,\\
\rho_{\zeta} (g_u) (e_i) &=&  e_{i+1},
\end{eqnarray*}
where $(e_i)$ is the canonical basis of $\Cset^{\Fset}$. Observe that
\begin{eqnarray*}
 \rho_{\zeta} (M(0,0,1)) (e_i) &=& \rho_{\zeta}( g_r g_u g_r^{-1} g_u^{-1})(e_i) \\
&=&\zeta e_i  .
\end{eqnarray*}
All irreducible representations except for the trivial one are of complex type, the complex conjugacy being given by $(m,n) \leftrightarrow(-m,-n)$ , $\zeta \leftrightarrow \bar \zeta$.

\smallskip

The action of $\textrm {Out} (G)$ on ${\mathcal Irr}_{\Cset} (G)$ is as follows.
For $(m,n) \in \Fset^2$ and $\varphi \in \textrm {Aut} (G)$, one has
$$  \chi_{m,n} \circ \varphi = \chi_{m',n'} ,  \quad (m',n') = (m,n).r(\varphi)\;.$$
For a primitive root $\zeta$ of order $p$, one has $\rho_{\zeta} \circ \varphi \simeq \rho_{\zeta'}$ with $\zeta' =
\zeta^{\det r(\varphi)}$.

\smallskip

Finally, it is easy to check that the morphism
$$ \textrm{Aff}(\mathcal O) / \textrm{Aut}(\mathcal O) \simeq GL(M)=GL(2,\Zset) \longrightarrow \textrm{Out}(G) \simeq GL(2,\Fset)$$
induced by the conjugacy in the affine group is simply the reduction modulo p.

\smallskip

We conclude that the image of $\textrm{Aff}(\mathcal O)$ in $ \textrm{Out}(G)$ is the group of matrices with
determinant $\pm 1$. It acts nontrivially on the set of $1$-dimensional irreducible representations of $G$.
However, these representations do not appear in the decomposition of $H_1^{(0)}(\mathcal O , \Cset)$. On the other hand,
the representation $\rho_{\zeta} $ is transformed either in itself or its complex conjugate, so every
{\bf real} irreducible representation of dimension $2p$ is fixed by the action of the affine group.

\smallskip

{\bf Question:} Is there a regular origami $M$ for which the action of the affine group on the set
of isotypic components of $H_1^{(0)}(M, \Rset)$ is nontrivial?

\end{document}